\documentclass[11pt,reqno]{amsart}

\calclayout

\usepackage{tikz} 
\usepackage{tkz-graph}
\tikzstyle{vertex}=[circle, fill, draw, inner sep=0pt, minimum size=6pt]
\newcommand{\vertex}{\node[vertex]}
\usetikzlibrary{shapes,backgrounds,patterns} 
\usepackage{times}
\usepackage[T1]{fontenc}
\usepackage{mathrsfs}
\usepackage{latexsym}
\usepackage{verbatim,url}
\usepackage{mathtools}

\usepackage{amsmath,amsfonts,amsthm,amssymb,amscd}
\input amssym.def
\input amssym.tex
\usepackage{color}

\usepackage{cite}
\nocite{*}

\newcommand\bea{\begin{eqnarray}}
\newcommand\eea{\end{eqnarray}}
\newcommand\bi{\begin{itemize}}
\newcommand\ei{\end{itemize}}
\newcommand\ben{\begin{enumerate}}
\newcommand\een{\end{enumerate}}

\newtheorem{thm}{Theorem}[section]

\newtheorem{cor}[thm]{Corollary}
\newtheorem{lem}[thm]{Lemma}
\newtheorem{prop}[thm]{Proposition}

\theoremstyle{definition}
\newtheorem{exa}[thm]{Example}
\newtheorem{defn}[thm]{Definition}

\newtheorem{rem}[thm]{Remark}
\newtheorem{que}[thm]{Question}

\newcommand{\bthm}{\begin{thm}}
\newcommand{\ethm}{\end{thm}}
\newcommand{\blem}{\begin{lem}}
\newcommand{\elem}{\end{lem}}

\newcommand{\conv}{\mbox{conv}}

\newcommand{\interior}{{\rm int}(\sigma)}

\newcommand{\R}{\ensuremath{\mathbb{R}}}

  \newcommand{\Lk}{\rm Lk}

\newcommand{\sm}{\setminus}

\numberwithin{equation}{section}

\begin{document}
\title[Neural codes, decidability, and a new local obstruction to convexity]{Neural codes, decidability, and \\ a new local obstruction to convexity}

\author{Aaron Chen}
\address{Aaron Chen, Department of
  Mathematics, Cornell University, Ithaca NY 14853}
\curraddr{University of Chicago, Chicago IL 60637}
\email{aaronchen@uchicago.edu} 
\author{Florian Frick}
\address{Florian Frick, Department of
  Mathematics, Cornell University, Ithaca NY 14853}
\curraddr{Department of Mathematical Sciences, Carnegie Mellon University, Pittsburgh PA 15213}
\email{frick@cmu.edu}

\author{Anne Shiu}
\address{Anne Shiu, Department of Mathematics, Texas A\&M University, College Station TX 77843}
\email{annejls@math.tamu.edu}

\date{\today}

\begin{abstract}
Given an intersection pattern of arbitrary sets in Euclidean space, 
is there an arrangement of convex open sets in
Euclidean space that exhibits the same intersections? 
This question is combinatorial and topological in nature, but is motivated
by neuroscience. Specifically, we are interested in a type of neuron called
a place cell, which fires precisely when an organism is in a certain region,
usually convex, called a place field. The earlier question, therefore, can
be rephrased as follows: Which neural codes, that is, 
patterns of neural activity, can arise from a collection of convex open sets? 
To address this question, Giusti and Itskov
proved that convex neural codes have no ``local obstructions,'' which are
defined via the topology of a code's simplicial complex. Codes without local
obstructions are called locally good, 
because the obstruction precludes the code from 
encoding the intersections of open sets that form a good cover.
In other words,
every good-cover code is locally good. Here we prove the converse: Every
locally good code is a good-cover code. We also prove that the good-cover
decision problem is undecidable. Finally, we reveal a stronger type of local
obstruction that prevents a code from being convex, and prove that the
corresponding decision problem is NP-hard. Our proofs use combinatorial and
topological methods.

   \end{abstract}

\maketitle

\section{Introduction}
This work addresses the following question: Which binary codes arise from the regions cut out by a collection of convex open sets in some Euclidean space?  One such code is 
$$\mathcal{C} =
\{1110,0111,1100,0110,1010,0011,1000,0100,0010,0001 \}=
\{123,234,12,23,13,24,34,1,2,3,4\}~,$$
which arises from the following convex open sets $U_1$, $U_2$, $U_3$, and $U_4$:
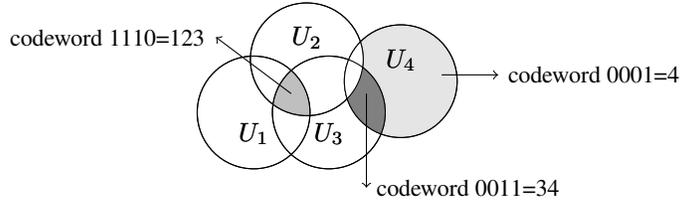
\begin{center}
\begin{tikzpicture}[scale=0.5] 

\def\firstcircle{(0,0) circle (1.5cm)} 
\def\secondcircle{(45:2cm) circle (1.5cm)} 
\def\thirdcircle{(0:2cm) circle (1.5cm)} 
\def\fourthcircle{(12:4cm) circle (1.5cm)} 

    \draw \firstcircle node[below] {\small $U_1$}; 
    \draw \secondcircle node [above] {\small $U_2$}; 
    \draw \thirdcircle node [below] {\small $U_3$}; 
    \draw \fourthcircle node [above] {\small $U_4$}; 

 \begin{scope}[fill opacity=.5] 
      \clip \secondcircle;
      \clip \firstcircle; 
      \fill[gray] \thirdcircle; 
    \end{scope} 
    
    \begin{scope} 
      \clip \thirdcircle; 
      \clip \secondcircle; 
      \fill[white] \fourthcircle; 
    \end{scope} 
    
     \begin{scope}[fill opacity=.2] 
      \fill[gray] \fourthcircle; 
    \end{scope} 
    
	 \begin{scope}[fill opacity=1] 
      \clip \thirdcircle; 
      \fill[gray] \fourthcircle; 
    \end{scope}     
    
     \begin{scope} 
      \clip \secondcircle; 
      \fill[white] \fourthcircle; 
    \end{scope} 

\draw \firstcircle node[below] {\small $U_1$}; 
    \draw \secondcircle node [above] {\small $U_2$}; 
    \draw \thirdcircle node [below] {\small $U_3$}; 
    \draw \fourthcircle node [above] {\small $U_4$}; 
\draw[->](3,.5) -- (3,-2) node[anchor=west] {\footnotesize codeword 0011=34} ; 
\draw[->](1,.5) -- (-1,2) node[anchor=east] {\footnotesize codeword 1110=123} ; 
\draw[->](5,1) -- (6.5,1) node[anchor=west] {\footnotesize codeword 0001=4} ; 
\end{tikzpicture} 
\end{center}
Here, some of the regions are labeled by the corresponding {\em codewords} in $\mathcal{C}$.  We can view each codeword as a vector in $\{0,1\}^4$ or as its support set, here a subset of $\{1,2,3,4\}$.

A closely related question is: Which 
{\em intersection patterns} arise from a collection of convex sets?  This question asks only which sets $U_i$ intersect, and not whether, for instance, there is a region where $U_1$ and $U_2$ intersect outside of $U_3$.
This topic -- intersection patterns of convex sets -- has been studied extensively 
(see~\cite{tancer-survey} for an overview), but the first question we posed has caught attention only recently~\cite{just-convex,no-go,what-makes,neural_ring,GB,new-alg,sparse,LSW,giusti-preprint,IKR,williams,zvi-yan}.  

The recent interest in this area is motivated by neuroscience, 
specifically from the study of neurons called {\em place cells}. 
The discovery of place cells by O'Keefe \emph{et al.} in 1971 was a major breakthrough that led to a shared 2014 Nobel Prize in Medicine or Physiology ~\cite{Oke1}. A place cell encodes spatial information about an organism's surroundings by firing precisely when the organism is in the corresponding {\em place field}.  In this context, a codeword represents the  neural firing pattern that occurs when the organism is in the corresponding region of its environment: 
the $i$th coordinate is 1 if and only if 
the organism is in the place field of neuron $i$. 
The resulting set of codewords is called a {\em neural code}.

Place fields can be modeled by convex open sets~\cite{neural_ring}, so 
we are interested in the following restatement of 
the question that opened this work: Which neural codes can arise from a collection of convex open sets?
To address this problem, Giusti and Itskov identified a {\em local obstruction}, defined via the topology of a code's simplicial complex, and proved that convex neural codes have no local obstructions~\cite{no-go}.  Codes without local obstructions are called {\em locally good}, as the obstruction prevents the code 
from encoding the intersections of open sets that form a good cover
(for instance, if the sets are convex).  
If such a good cover exists (for instance, from a collection of convex open sets), then the code is a {\em good-cover} code.  Thus, we have:
\begin{center}
$\mathcal{C}$ is convex 
$\quad \Rightarrow \quad$
$\mathcal{C}$ is a good-cover code 
$\quad \Rightarrow \quad$
$\mathcal{C}$ is locally good.
\end{center}
The converse of the first implication is false~\cite{LSW}.
The second implication is the starting point of our work.  
We prove that the implication is in fact an equivalence:
every locally good code is a good-cover code (Theorem~\ref{thm:iff}).
We also prove that the good-cover decision problem is undecidable (Theorem~\ref{thm:undecidable}).  

Next, we discover a new, stronger type of local obstruction that precludes a code from being convex (Theorem~\ref{thm:new-obstruction}).
Like the prior obstruction, the new obstruction is defined in terms of a code's simplicial complex, 
but in this case the link of ``missing'' codewords must be ``collapsible'' 
(which is implied by ``contractible'', the condition in the original type of obstruction).  We call codes without 
the new obstruction {\em locally great}, and examine the corresponding decision problem.  We prove that
the locally-great decidability problem is decidable, and in fact NP-hard (Theorem~\ref{thm:decidable}).

Thus, our results refine the implications we saw earlier, as follows:
\begin{center}
$\mathcal{C}$ is convex 
$\quad \Rightarrow \quad$
$\mathcal{C}$ is locally great 
$\quad  \Rightarrow \quad$
$\mathcal{C}$ is a good-cover code 
$\quad \Leftrightarrow \quad$
$\mathcal{C}$ is locally good. \\
{(NP-hard problem)} 
\quad 
\quad 
\quad \quad \quad \quad 
(undecidable)
\quad \quad \quad \quad \quad \quad
\end{center}
Finally, we add another implication to the end of those listed above, by noting that every locally good code can be realized by connected open sets, but not vice-versa (Proposition~\ref{prop:connected}).
Taken together, our results resolve fundamental questions in the theory of convex neural codes.

The outline of our work is as follows.
Section~\ref{sec:bkgrd} provides background on neural codes, local obstructions, and criteria for convexity. 
In Sections~\ref{sec:iff}--\ref{sec:new-obs}, we prove the results listed above, using classical tools from topology and combinatorics.
Finally, our discussion in Section~\ref{sec:discussion} lists open questions arising from our work.

\section{Background} \label{sec:bkgrd}
Here we introduce notation as well as basic definitions in the theory of neural codes. 

We define $[n] := \{1,2,\dots,n\}$. We will reserve lowercase Greek letters (e.g., $\sigma$ and $\tau$)
to denote subsets of $[n]$ (for some $n$). 
Such a subset usually refers to a codeword in a neural code
(Definition~\ref{def:code}) 
or a face in a simplicial complex (Definition~\ref{def:simplicial-cpx}). 
For shorthand, we will omit the braces and commas; e.g., if $\tau = \{1,2,3\}$ and $\sigma = \{2,3,4\}$, we write $\tau = 123$, $\sigma = 234,$ and $\tau \cap \sigma = 23$. 

\subsection{Codes, simplicial complexes, and the nerve theorem}
Given a collection $\mathcal{U} = \{U_1,U_2,\dots,U_n\}$ of 
sets (place fields) in some {\em stimulus space} $X \subseteq \R^d$ 
 and some $\tau \subseteq [n]$, let $U_\tau := \bigcap_{i \in \tau} U_i$, where $U_{\emptyset}:=X$. 

\begin{defn} \label{def:code} ~
\begin{enumerate}
\item A \textbf{neural code} $\mathcal{C}$ on $n$ neurons is a subset of $2^{[n]}$, and each $\sigma \in \mathcal{C}$ is a \textbf{codeword}. Any codeword that is maximal in $\mathcal{C}$ with respect to set inclusion is a $\textbf{maximal codeword}$.
\item A code $\mathcal{C}$ is \textbf{realized} by a collection of sets $\mathcal{U} = \{U_1,U_2,\dots,U_n\}$ in a stimulus space $X \subseteq \mathbb{R}^d$ if 
\begin{equation} \label{realize}
\sigma \in \mathcal{C} ~ \Longleftrightarrow ~ U_\sigma \setminus \bigcup_{j \notin \sigma} U_j \neq \emptyset.
\end{equation}
Conversely, given a collection of sets $\mathcal{U}$, let  $\mathcal{C}(\mathcal{U})$ denote the unique code realized by $\mathcal{U}$, via~\eqref{realize}.
\end{enumerate}
\end{defn}

\begin{rem}
	A neural code is often referred to as a 
    {\em hypergraph} in the literature. Also, given a collection of subsets $U_1,U_2, \dots, U_n \subseteq X$, the neural code that is realized by these subsets can be defined as the collection of sets $U(x) = \{i \in [n] \ | \ x \in U_i\}$, where $x$ varies over~$X$. 
\end{rem}

Every neural code can be realized by open sets~\cite{neural_ring} 
and by convex sets~\cite{just-convex}. 
We are interested, however, in realizing neural codes by 
sets that are both open and convex.
This is because 
(biological) place fields
are approximately convex and have positive measure, 
which are properties captured by convex open sets.
\begin{defn}
A neural code $\mathcal{C}$ is: \begin{enumerate}
\item \textbf{convex} if $\mathcal{C}$ can be realized by a collection of convex open sets. \item a \textbf{good-cover code} if $\mathcal{C}$ can be realized by 
a collection of 
contractible open sets $\mathcal{U}$ such that every nonempty intersection of sets in $\mathcal{U}$ is also contractible. Such a $\mathcal{U}$ is called a \textbf{good cover}.
\item \textbf{connected} if $\mathcal{C}$ can be realized by a collection of connected open sets.
\end{enumerate}
\end{defn}

\begin{exa} We revisit the code $\mathcal{C} = \{123,234,12,23,13,24,34,1,2,3,4,\emptyset\}$ from the Introduction, where we saw that $\mathcal{C}$ is convex.
Hence, $\mathcal C$ is a good-cover code and also a connected code.  
\end{exa}

Connected codes were classified recently by Mulas and Tran (not every code is connected)~\cite{connected-codes}.  Good-cover codes are connected, but not vice-versa (see Proposition~\ref{prop:connected} later in this section).

\begin{rem}[Codes and the empty set] \label{rmk:empty-set}
A code $\mathcal{C}$ is convex (respectively, a good-cover code) if and only if $\mathcal{C} \cup \{ \emptyset \}$
is convex (respectively, a good-cover code) \cite{what-makes}.
Indeed, if $\emptyset \notin \mathcal{C}$ and $\mathcal{C}$ is realized by convex open sets (respectively, a good cover) $\mathcal{U}$,
then $\mathcal{C} \cup \emptyset$ is realized by $\{U \cap B \mid U \in \mathcal{U} \}$,
where ${B}$ is an open ball that contains a point from each region cut out by $\mathcal{U}$.
Conversely, if $\emptyset \notin \mathcal{C}$ and 
$\mathcal{C} \cup \emptyset$ is realized by $ \mathcal{U} $ (with respect to some stimulus space),
then the code $\mathcal{C}$ is realized by $\mathcal{U}$
with respect to the stimulus space  $X=\cup_{U \in \mathcal{U}} U $.
\end{rem}

Convex codes on up to four neurons have been classified~\cite{what-makes}.  This classification was enabled by analyzing codes according to the simplicial complex they generate (see Definition~\ref{def:s-cpx-code}).

\begin{defn} \label{def:simplicial-cpx}
An \textbf{abstract simplicial complex} $\Delta$ on $[n]$ is a subset of $2^{[n]}$ that is closed under taking subsets. Each $\sigma \in \Delta$ is a \textbf{face} of $\Delta$.  The \textbf{facets} of $\Delta$ are the faces that are maximal with respect to inclusion.  The \textbf{dimension} of a face $\sigma$ is $|\sigma|-1$, and the \textbf{dimension} of a simplicial complex $\Delta$, denoted by $\dim \Delta$, is the 
maximum dimension of the faces of $\Delta$.
\end{defn}
Every simplicial complex $\Delta$ can be realized geometrically in a Euclidean space of sufficiently high dimension, and we let $|\Delta|$ denote such a geometric realization (which is unique up to homeomorphism). Note that the dimension of a simplicial complex matches the dimension of its realization: $\dim \Delta = \dim |\Delta|$.

\begin{defn}
For a face $\sigma$ of a simplicial complex $\Delta$,
the \textbf{restriction} of $\Delta$ to $\sigma$ is the simplicial complex
	\[
    \Delta|_{\sigma} ~:=~ \{ \tau \in \Delta \mid \tau \subseteq \sigma \}~.
    \]
The \textbf{link} of $\sigma$ in $\Delta$ is the simplicial complex:
$$
\Lk_\sigma(\Delta) ~:=~ \{\tau \in \Delta \mid \sigma \cap \tau = \emptyset \mbox{ and } \sigma \cup \tau \in \Delta\}~.
$$
\end{defn}
\noindent
Links are usually written as ${\rm Lk}_{\Delta}(\sigma)$, instead of $\Lk_\sigma(\Delta)$, but, following \cite{what-makes}, we prefer to have $\sigma$ in the subscript, because we often consider the link in several simplicial complexes.
Note that $\sigma \cup \tau \in \Delta$ is the same as saying that $\sigma * \tau \subseteq |\Delta|$, where $*$ is the topological join.

\begin{defn}\label{def:cone}
Let $\Delta$ be a simplicial complex on $[n]$.  The {\bf cone over $\Delta$} on a new vertex $v$ is the following simplicial complex on $[n] \dot \cup \{v\}$:
\[
{\rm cone}_v(\Delta) ~:=~
	\{\sigma \cup \{v\} \mid \sigma \in \Delta \} \cup \Delta~.
\]
\end{defn}
\noindent
By construction, 
${\rm Lk}_v(
{\rm cone}_v(\Delta) ) = \Delta$.
We will use this fact throughout our work.

\begin{defn} \label{def:s-cpx-code}
For a code $\mathcal{C}$ on $n$ neurons, 
the {\bf simplicial complex of} $\mathcal{C}$, denoted by $\Delta(\mathcal{C})$,
 is the smallest simplicial complex on $[n]$ that contains $\mathcal{C}$.  
\end{defn}
\noindent
Note that for $\mathcal{C}=\mathcal{C}(\mathcal{U})$, we have $\sigma \in \Delta(\mathcal{C})$ if and only if $U_{\sigma} \neq \emptyset$ (cf.\ \eqref{realize}).
Also, the facets of $\Delta(\mathcal{C})$ are the maximal codewords of $\mathcal{C}$.

\begin{exa}
The code $\mathcal{C} =\{12,23,1,2,3,\emptyset\}$ is convex and realized here:

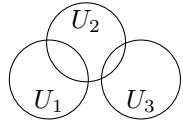
\begin{center}
\begin{tikzpicture}[scale=0.35] 
\def\firstcircle{(0,0) circle (1.5cm)} 
\def\secondcircle{(45:2cm) circle (1.5cm)} 
\def\thirdcircle{(0:3.5cm) circle (1.5cm)} 
    \draw \firstcircle node[below] {\small $U_1$}; 
    \draw \secondcircle node [above] {\small $U_2$}; 
    \draw \thirdcircle node [below] {\small $U_3$}; 
\end{tikzpicture} 
\end{center}
The simplicial complex of the code, $\Delta(\mathcal{C})$, is realized here:
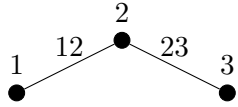
\begin{center}
    \begin{tikzpicture}[scale=0.7]
        \vertex[label=$1$](p1) at (-2,0) {};
        \vertex[label=$2$](p2) at (0,1) {};
        \vertex[label=$3$](p3) at (2,0) {};
  \path [-] (p1) edge node[above] {$12$} (p2);
   \path [-](p2) edge node[above] {$23$} (p3);
    \end{tikzpicture}
\end{center}
\end{exa}

A related notion to $\Delta(\mathcal{C})$ is the nerve of a cover (see Remark~\ref{rmk:nerve}), which we define now. 
\begin{defn} Given a collection of sets $\mathcal{U} = \{ U_1,U_2,\dots,U_n\}$, the \textbf{nerve} of $\mathcal{U}$, denoted by $\mathcal{N}(\mathcal{U})$, is the simplicial complex on $[n]$ defined by:
$$
\sigma \in \mathcal{N}(\mathcal{U}) ~ \Longleftrightarrow ~ U_\sigma \neq \emptyset~.
$$
\end{defn}
\begin{rem} \label{rmk:nerve} $\mathcal{N}(\mathcal{U}) = \Delta(\mathcal{C}(\mathcal{U}))$.
\end{rem}

Next, we recall the classical result called the nerve theorem or nerve lemma~\cite{weil}.  The version we state is~\cite[Corollary~4G.3]{hatcher}:
\begin{prop}[Nerve theorem] \label{prop:nerve-thm}
If  $\mathcal{U}$ is a finite collection of nonempty and contractible
open sets 
that cover a paracompact space $S$
such that every intersection of sets is either empty or contractible, then 
$S$ is homotopy equivalent to $\mathcal{N}(\mathcal{U})$.  
\end{prop}

Metric spaces are paracompact~\cite{stone}, so good-cover realizations of codes satisfy the hypotheses of Proposition~\ref{prop:nerve-thm}.  
Thus, if we determine that a code does not satisfy the conclusion of this proposition,
then we conclude it is not convex.  We turn to this topic now.

\subsection{Local obstructions and criteria for convexity} \label{sec:local-obs}
One way to detect non-convexity of a neural code is to find what is known as a local obstruction.
\begin{defn} \label{def:local-obs}
Given a code $\mathcal{C}$ that is realized by open sets $\mathcal{U}$, a \textbf{local obstruction} is a pair $(\sigma,\tau)$
for which 
$$
U_\sigma ~ \subseteq ~ \bigcup_{i \in \tau} U_i~,
$$
where $\tau \neq \emptyset$ and $\Lk_\sigma(\Delta(\mathcal{C})|_{\sigma \cup \tau})$ is {\em not} contractible.
A code $\mathcal{C}$ with {\em no} local obstructions is \textbf{locally good}. 
\end{defn}
\noindent
Definition~\ref{def:local-obs} does not depend on the choice of open sets $\mathcal{U}$~\cite{what-makes}.

The name ``local obstruction''
is due to the following result, which states that if a code has a local obstruction, then it is not a good-cover code, and therefore not convex.
\begin{prop} [Giusti and Itskov~\cite{no-go}; Curto {\em et al.}~\cite{what-makes}] \label{prop:locally-good}
\begin{equation} \label{relation}
\mathcal{C} \mbox{ is convex } \Rightarrow \mathcal{C} \mbox{ is a good-cover code } \Rightarrow \mathcal{C} \mbox{ is locally good}.
\end{equation}
\end{prop}

A natural question is:  
Do the converses of the implications in~\eqref{relation} hold? 
For the second implication, the converse is true; we will prove this in the next section  (Theorem~\ref{thm:iff}).  

As for the first implication, the converse is in general false (but true for codes in which all codewords have size at most two~\cite{sparse} and codes on up to four neurons~\cite{what-makes}).
The first counterexample, which is a code on five neurons, was found by Lienkaemper, Shiu, and Woodstock~\cite{LSW}:
\begin{prop} [Counterexample code~\cite{LSW}] \label{prop:counterexample-code}
The following neural code is a good-cover code, but non-convex: $\mathcal{C} = \{2345,123,134,145,13,14,23,34,45,3,4,\emptyset\}$.
\end{prop}
\noindent
This code, it turns out, {is} realizable by {\em closed} convex sets instead of open convex sets~\cite{giusti-preprint}. 

Returning to the topic of detecting local obstructions,
the next result gives a way to do so that is more efficient than simply applying the definition (Definition~\ref{def:local-obs}).
As it turns out, we need to check only the links of faces that are intersections of facets of $\Delta(\mathcal{C})$.
\begin{defn} 
The set of \textbf{mandatory codewords of 
a simplicial complex} $\Delta$ is
$\mathcal{M}(\Delta) := \{\sigma \in \Delta \mid {\rm Lk}_\sigma(\Delta)$ is not contractible$\}$; and the 
set of \textbf{non-mandatory codewords of} $\Delta$ is $\Delta \setminus \mathcal{M}(\Delta)$.
The set of \textbf{mandatory codewords of a code} $\mathcal{C}$ 
is $\mathcal{M}(\Delta(\mathcal{C}))$.
\end{defn}
\begin{prop} [Curto {\em et al.}~\cite{what-makes}] \label{prop:mandatory}
A code $\mathcal{C}$ is locally good if and only if it contains all its mandatory codewords (i.e., $\mathcal{M}(\Delta(\mathcal{C})) \subseteq \mathcal{C}$). Also, every mandatory codeword is an intersection of maximal codewords.
\end{prop}

As a corollary, {\em max-intersection complete} codes, 
those that are closed under taking intersections of maximal codewords, are locally good. 
In fact, these codes are convex~\cite{giusti-preprint}.
Note that the (non-convex) counterexample code $\mathcal{C}$ from Proposition~\ref{prop:counterexample-code} is {\em not} max-intersection complete, because the intersection of maximal codewords $123 \cap 134\cap 145=1$ is missing from $\mathcal{C}$.

Another result pertaining to convexity,
due to Cruz {\em et al.}~\cite{giusti-preprint}, is as follows: for codes with the same simplicial complex, convexity is a monotone property with respect to inclusion. 
That is, {\em if $\mathcal{C}$ is convex, and $\mathcal{C} \subseteq \mathcal{C}'$ where $\Delta(\mathcal{C}) = \Delta(\mathcal{C}')$, then $\mathcal{C}'$ is convex.}

We end this section by showing that good-cover codes are realizable by connected open sets, but not vice-versa:
\begin{prop} \label{prop:connected}
Every good-cover code is connected, but not every connected code is a good-cover code.
\end{prop}
\begin{proof}
Contractible sets are connected, so, by definition, 
good-cover codes are connected.   
As for the converse, consider the following code: 
 $\mathcal{C} = \{124, 134, 234, 14, 24, 34, \emptyset \}$.  The codeword $4$,
 which is not in $\mathcal C$, 
 is a mandatory codeword of $\mathcal C$, 
 as  
 ${\rm Lk}_{\{4\}}(\Delta (\mathcal{C}))$ is 
 the following non-contractible simplicial complex:

 
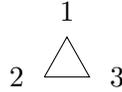
\begin{center}
  \begin{tikzpicture}[scale=0.6]
	\draw (0,0)--(1,0)--(0.5,0.886)--(0,0);
       \node[label=above:$1$] at (0.5,0.8) {};
       \node[label=left:$2$] at (0,0) {};
       \node[label=right:$3$] at (1,0) {};
	\end{tikzpicture}
\end{center}
Thus, by definition, $\mathcal{C}$ is not locally good, and so, by Proposition~\ref{prop:mandatory}, is not a good-cover code.  We complete the proof by displaying the following realization of $\mathcal{C}$ by connected, open sets in $\mathbb{R}^2$:
 
 
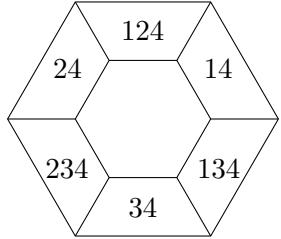
\begin{center}
 \begin{tikzpicture}[scale=0.9]
	\draw (1,0)--(0.5,0.866)--(-0.5,0.866)--(-1,0)--(-0.5,-0.866)--(0.5,-0.866)--(1,0);
	\draw (2,0)--(1,1.73)--(-1,1.73)--(-2,0)--(-1,-1.73)--(1,-1.73)--(2,0);
   	\draw (1,0)--(2,0);
	\draw (0.5,0.866)--(1,1.73);
	\draw (-0.5,0.866)--(-1,1.73);
	\draw (-1,0)--(-2,0);
	\draw (-0.5,-0.866)--(-1,-1.73);
	\draw (0.5,-0.866)--(1,-1.73);
    \node[label=above:$124$] at (0,0.866) {};
    \node[label=below:$34$] at (0,-0.866) {};
    \node[label=above right:$14$] at (0.6,0.3) {};
    \node[label=above left:$24$] at (-0.6,0.3) {};
    \node[label=below left:$234$] at (-0.5,-0.3) {};
    \node[label=below right:$134$] at (0.5,-0.3) {};
\end{tikzpicture}
 \end{center}
More precisely, for $i=1,2,3,4$, consider (the closures of) the regions
above that are labeled by some $\sigma$ for which $i \in \sigma$.
Now let $U_i$ be the interior of the union of all such regions.
\end{proof}

\section{Locally good codes are good-cover codes} \label{sec:iff}
In this section we will prove that being locally good is equivalent to being a good-cover code (Theorem~\ref{thm:iff}). 
We accomplish this by constructing a good-cover realization of any locally good code.  Our construction has two steps.  We first realize {\em any} code via (not necessarily open) subsets of a 
geometric realization of its simplicial complex (Proposition~\ref{prop:realization}), and then do a ``reverse deformation retract''
to obtain a realization by open sets
(Proposition~\ref{prop:realization-U}).

\subsection{Code-complex realizations}
The idea behind the following construction, which realizes any code $\mathcal{C}$ by subsets of $|\Delta(\mathcal{C})|$,
is to realize codewords in $\mathcal{C}$ by the corresponding faces of $|\Delta(\mathcal{C})|$.
Thus, we will simply delete faces corresponding to codewords that are {\em not} in $\mathcal{C}$.  
Accordingly, 
for any simplicial complex~$\Delta$ and any $\sigma \in \Delta$,
let $\interior$ denote the relative interior of the realization of the face $\sigma$ within $|\Delta|$ 
(if $v$ is a vertex, then ${\rm int}(v)=\{v\}$).  See the left-hand side of Figure~\ref{fig:R}.
It follows that $|\Delta| = \dot{\cup}_{\emptyset \neq \sigma \in \Delta} \interior$.  Thus, $|\Delta|$ is a CW-complex built from the $\interior$'s.

\begin{defn} \label{def:code-cpx-realization}
Let $\mathcal{C}$ be a code on $n$ neurons.
For each $i \in [n]$, consider the following subset of $|\Delta(\mathcal{C})|$: 
	\begin{align*}
		V_i ~:=~ \bigcup_{ i \in \sigma \in \mathcal{C}} \interior ~.
	\end{align*}
Then $\{V_1, V_2, \dots, V_n \}$ is the {\bf code-complex realization} of $\mathcal{C}$.
\end{defn}

\begin{figure}[ht]
\begin{tikzpicture}[scale=1]
    \tikzstyle{point}=[circle]
    \node(a)[point] at (0,0) {};
    \node[label=left:$(3)$] [point] at (0.1,0) {};
    \node (b)[point] at (2,0) {};
    \node[label=right:$(2)$] [point] at (1.8,0) {};
    \node (c)[point] at (1,1.7) {};
    \node[label=above:$(1)$] [point] at (1,1.5) {};
	\fill[gray!30] (0,0)--(2,0)--(1,1.7);
    \draw   (a.center) -- (b.center)node[pos=.5,below]{$(23)$}node[pos=.5,above,yshift = .2cm]{$(123)$} -- (c.center)node[pos=.5,right]{$(12)$} -- (a.center)node[pos=.5,left]{$(13)$} ;
    
    \node (d)[point] at (4,0) {};
    \node (e)[point] at (6,0) {};
    \node (f)[point] at (5,1.7) {};
   
    \node (g)[point] at (3,0) {};
    \node (h)[point] at (3.5,-0.85) {};
   
    \node (i)[point] at (7,0) {};
    \node (j)[point] at (6.5,-.85) {};

    \node (k)[point] at (4.5,2.55) {};
    \node (l)[point] at (5.5,2.55) {};

    \draw   (d.center) -- (e.center)node[pos=.5,below]{$R_{23}$}node[pos=.5,above,yshift = .2cm]{$R_{123}$} -- (f.center)node[pos=.5,right]{$R_{12}$} -- (d.center) node[pos=.5,left]{$R_{13}$};
     
     \draw[dashed]   (g.center) -- (d.center)node[pos=.5,below] {$R_3$} -- (h.center);

     \draw[dashed]   (i.center) -- (e.center)node[pos=.5,below] {$R_2$} -- (j.center);

     \draw[dashed]   (k.center) -- (f.center) -- (l.center)node[pos=0,above,yshift=.3cm]{$R_1$};
\end{tikzpicture}
\caption{The correspondence between the face-interiors
$\interior$ that make up the $2$-simplex and the regions $R_{\sigma}$ 
that deformation retract to the $\interior$'s (see equation~\eqref{eq:R}). 
We define the region $R_{12}$ to contain the face-interior labeled by $(12)$,
and $R_1$ to contain  
the dashed lines of its boundary 
and the vertex $(1)$
(so it is closed), which is why 
in Proposition~\ref{prop:realization-U} 
we must pass from the unions of regions $R_{\sigma}$ to 
their relative interiors.
\label{fig:R}}
\end{figure}
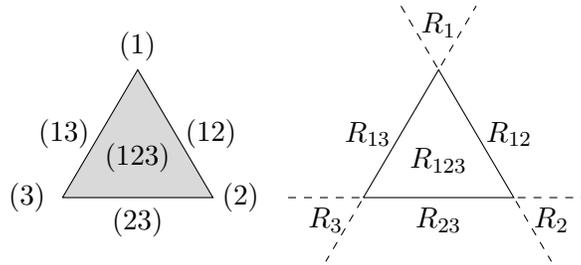

\begin{prop} \label{prop:realization}
Let $\mathcal{C}$ be a code on $n$ neurons.
Then the code-complex realization $\{V_1, V_2, \dots, V_n \}$
of $\mathcal{C}$ given in 
Definition~\ref{def:code-cpx-realization} 
 is a realization of $\mathcal{C}$.
Moreover, if $\mathcal{C}$ is locally good, then every nonempty intersection of sets from $\{V_1, V_2, \dots, V_n \}$ is contractible.
\end{prop}

We postpone the proof of Proposition~\ref{prop:realization} to the end of this subsection.

\begin{rem} \label{rmk:not-s-cpx}
Code-complex realizations 
(Definition~\ref{def:code-cpx-realization})
are in general not simplicial complexes or even CW-complexes.
See Examples~\ref{ex:realize-code} and~\ref{ex:realize-code-2},
where the sets $V_i$ and their unions are not even closed.
\end{rem}

\begin{rem}
Code-complex realizations 
(Definition~\ref{def:code-cpx-realization})
are somewhat similar to prior constructions that also specify regions that correspond
to codewords~\cite{what-makes,giusti-preprint} or intersection patterns~\cite{tancer-survey}.
\end{rem}

\begin{exa} \label{ex:realize-code}
Consider the (locally good) code $\mathcal{C} = \{123,12,23,1,2,\emptyset\}$. The 
code-complex realization of $\mathcal{C}$ given in 
Definition~\ref{def:code-cpx-realization} is depicted here,
along with $|\Delta(\mathcal{C})|$:

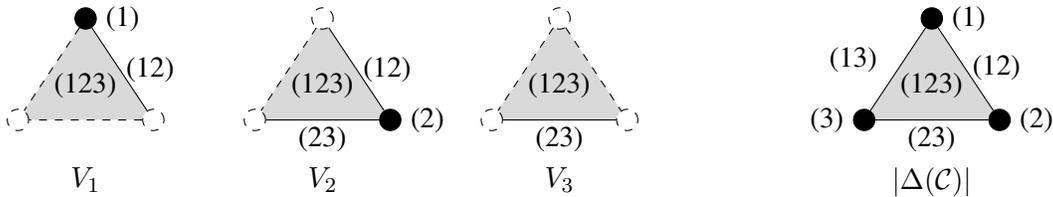
\begin{center}
\begin{tikzpicture}[scale=0.9]
\fill[gray!30] (-8,-3)--(-6,-3)--(-7,-1.5);
\draw node[fill=white, dashed, draw=black, inner sep=.1cm] (ov42) at (-8, -3) [shape=circle] {};
\draw node[fill=white, dashed, draw=black, inner sep=.1cm] (cv42) at (-6,-3) [shape=circle] {};
\draw node[fill=black, draw=black, inner sep=.1cm] (top-1) at (-7,-1.5) [shape=circle] {};
\draw[-] (top-1)--(cv42);
\draw[dashed] (cv42)--(ov42);
\draw[dashed] (top-1)--(ov42);
\node[label=right : {(1)}] at (-7,-1.5) {};
\node[label=right : {(12)}] at (-6.7,-2.25) {};
\node[label=below : {(123)}] at (-7,-2) {};
\node[label=below : {$V_1$}] (v42) at (-7,-3.4) {};
\fill[gray!30] (-4.5,-3)--(-2.5,-3)--(-3.5,-1.5);
\draw node[fill=white, dashed, draw=black, inner sep=.1cm] (ov) at (-4.5, -3) [shape=circle] {};
\draw node[fill=black, draw=black, inner sep=.1cm] (cv) at (-2.5,-3) [shape=circle] {};
\draw node[fill=white, dashed, draw=black, inner sep=.1cm] (top) at (-3.5,-1.5) [shape=circle] {};
\draw[-] (top)--(cv);
\draw[-] (cv)--(ov);
\draw[dashed] (top)--(ov);
\node[label=below : {(23)}] at (-3.5,-2.8) {};
\node[label=right : {(12)}] at (-3.2,-2.25) {};
\node[label=right : {(2)}] at (-2.5,-3) {};
\node[label=below : {(123)}] at (-3.5,-2) {};
\node[label=below : {$V_2$}] (v42) at (-3.5,-3.4) {};
\fill[gray!30] (-1,-3)--(1,-3)--(0,-1.5);
\draw node[fill=white, dashed, draw=black, inner sep=.1cm] (o) at (-1, -3) [shape=circle] {};
\draw node[fill=white, dashed, draw=black, inner sep=.1cm] (c) at (1,-3) [shape=circle] {};
\draw node[fill=white, dashed, draw=black, inner sep=.1cm] (t) at (0,-1.5) [shape=circle] {};
\draw[dashed] (t)--(c);
\draw[-] (c)--(o);
\draw[dashed] (t)--(o);
\node[label=below : {(23)}] at (0,-2.8) {};
\node[label=below : {(123)}] at (0,-2) {};
\node[label=below : {$V_3$}] (v42) at (0,-3.4) {};
\fill[gray!30] (4.5,-3)--(6.5,-3)--(5.5,-1.5);
\draw node[fill=black, draw=black, inner sep=.1cm] (left) at (4.5, -3) [shape=circle] {};
\draw node[fill=black, draw=black, inner sep=.1cm] (right) at (6.5,-3) [shape=circle] {};
\draw node[fill=black, draw=black, inner sep=.1cm] (high) at (5.5,-1.5) [shape=circle] {};
\draw[-] (left)--(right);
\draw[-] (high)--(right);
\draw[-] (left)--(high);
\node[label=right : {(1)}] at (5.5,-1.5) {};
\node[label=right : {(2)}] at (6.5,-3) {};
\node[label=left : {(3)}] at (4.5,-3) {};
\node[label=right : {(12)}] at (5.8,-2.22) {};
\node[label=left : {(13)}] at (5,-2.1) {};
\node[label=below : {(23)}] at (5.5,-2.8) {};
\node[label=below : {(123)}] at (5.5,-2) {};
\node[label=below : {$|\Delta(\mathcal{C})|$}] (v42) at (5.5,-3.4) {};
\end{tikzpicture}
\end{center}
\end{exa}

\begin{exa} \label{ex:realize-code-2}
For the (not locally good) code $\mathcal{C} = \{13,23,1,\emptyset\}$,
the 
code-complex realization
is:
\begin{center}
\begin{tikzpicture}[scale=0.9]
\draw node[fill=white, dashed, draw=black, inner sep=.1cm] (ov42) at (-8, -3) [shape=circle] {};
\draw node[fill=black, draw=black, inner sep=.1cm] (top-1) at (-7,-1.5) [shape=circle] {};
\draw[-] (top-1)--(ov42);
\node[label=right : {(1)}] at (-7,-1.5) {};
\node[label=right : {(13)}] at (-8.7,-2.2) {};
\node[label=below : {$V_1$}] (v42) at (-7,-3.1) {};
\draw node[fill=white, dashed, draw=black, inner sep=.1cm] (ov) at (-4.5, -3) [shape=circle] {};
\draw node[fill=white, dashed, draw=black, inner sep=.1cm] (cv) at (-2.5,-3) [shape=circle] {};
\draw[-] (cv)--(ov);
\node[label=above : {(23)}] at (-3.5,-3.2) {};
\node[label=below : {$V_2$}] (v42) at (-3.5,-3.1) {};
\draw node[fill=white, dashed, draw=black, inner sep=.1cm] (o) at (-1, -3) [shape=circle] {};
\draw node[fill=white, dashed, draw=black, inner sep=.1cm] (c) at (1,-3) [shape=circle] {};
\draw node[fill=white, dashed, draw=black, inner sep=.1cm] (t) at (0,-1.5) [shape=circle] {};
\draw[-] (t)--(o);
\draw[-] (c)--(o);
\node[label=above : {(23)}] at (0,-3.2) {};
\node[label=left : {(13)}] at (-0.3,-2.2) {};
\node[label=below : {$V_3$}] (v42) at (0,-3.1) {};
\end{tikzpicture}
\end{center}
\end{exa}

Comparing Examples~\ref{ex:realize-code} and~\ref{ex:realize-code-2}, note that, consistent with Proposition~\ref{prop:realization}, the sets $V_i$
in the first example are contractible and have contractible intersections, but not in the second example ($V_3$ is not connected).   

The proof of Proposition~\ref{prop:realization}
relies on the following lemma, which may be of independent interest.

\begin{lem} \label{lem:code-cpx}
Let $\Sigma$ be a simplicial complex, and let
$\Omega$ be a subset of the non-mandatory codewords of $\Sigma$.
Then, if $\Sigma$ is contractible, then so is the code-complex realization of the code $\Sigma \setminus \Omega$.
\end{lem}

We prove Lemma~\ref{lem:code-cpx} in Appendix~A.

\begin{proof}[Proof of Proposition~\ref{prop:realization}]
By construction, we have
	\begin{align*}
		\bigcup_{ i =1}^n V_i 
        ~=~ 
              \bigcup_{ \emptyset \neq \sigma \in \mathcal{C}}^{\cdot } \interior ~,
	\end{align*}
and the subset $\interior$ realizes the codeword $\sigma$.
If $\emptyset \notin \mathcal{C}$, define the stimulus space $X= \cup_{i=1}^n V_i$; otherwise, define $X$ to be the ambient 
Euclidean space. Thus, $\{V_1, V_2, \dots, V_n \}$ realizes $\mathcal{C}$ with respect to $X$.

Now assume
that $\mathcal{C}$ is locally good.
We must show that every intersection of the $V_i$'s is empty or contractible.
To this end, let $\emptyset \neq \tau \subseteq [n]$, and assume that $V_{\tau}$ is nonempty.  Note that 
\begin{align*}
V_{\tau} ~=~ \bigcup_{\tau \subseteq \sigma \in \mathcal{C}} \interior ~. 
\end{align*}
We consider two cases.
If $\tau \in \mathcal{C}$, then 
	$V_{\tau} = \mathrm{int}(\tau) \cup \bigcup_{\tau \subsetneq \sigma \in \mathcal{C}} \interior$.  Now 
    consider the deformation retract of $|\Delta^n|$ to the face $\overline{\tau}$, arising from orthogonal projection to that face.
    It is straightforward to check that this restricts to a deformation retraction of $V_{\tau}$ to $\mathrm{int}(\tau)$,
    which is contractible.  Thus, $V_{\tau}$ is contractible.
    
    Consider the remaining case, when $\tau \notin \mathcal{C}$.
Then $V_{\tau} = \bigcup_{\tau \subsetneq \sigma \in \mathcal{C}} \interior $.  
We will show that $V_{\tau}$ is contractible by 
proving first that $V_{\tau}$ is homotopy equivalent to the code-complex realization of the following {\em link of a code}~\cite[Definition 2.3]{giusti-preprint}
or, for short, ``code-link'':
\begin{align} \label{eq:code-lk}
 {\rm Lk}_{\tau}(\mathcal{C}) ~:=~
 \{ \eta \subseteq [n] \mid \tau \cap \eta = \emptyset \mbox{ and } \tau \cup \eta \in \mathcal{C} \}~,
\end{align}
and then proving that this code-complex realization, which we denote by $\mathcal{L}$,
is contractible.

To see that $V_{\tau} \simeq \mathcal{L}$, we compute:
\begin{align} \label{eq:V-tau}
V_{\tau} ~=~ \bigcup_{\tau \subsetneq \sigma \in \mathcal{C}} \interior
	~=~ \bigcup_{ \emptyset \neq \eta \in {\rm Lk}_{\tau}(\mathcal{C}) }^{\cdot} \mathrm{int}(\tau \ \dot{\cup} \ \eta)~, 
\end{align}
and
\begin{align} \label{eq:L}
\mathcal{L}~ =~ \bigcup_{ \emptyset \neq \eta \in {\rm Lk}_{\tau}(\mathcal{C}) }^{\cdot} \mathrm{int}(\eta)~.
\end{align}
It is straightforward to see, from~\eqref{eq:V-tau} and~\eqref{eq:L}, that $V_{\tau}$ deformation retracts to a copy of $\mathcal{L}$. 

To see that $ \mathcal{L}$
is contractible,
we will appeal to Lemma~\ref{lem:code-cpx} (where $\Sigma={\rm Lk}_{\tau} (\Delta(\mathcal{C})) $ and $\Sigma \setminus \Omega$ is the code-link in equation~\eqref{eq:code-lk}).  To apply Lemma~\ref{lem:code-cpx}, we must show that 
every $\sigma$ in ${\rm Lk}_{\tau} (\Delta(\mathcal{C}))$ but not in the code-link~\eqref{eq:code-lk} is {\em not} a mandatory codeword of ${\rm Lk}_{\tau} (\Delta(\mathcal{C}))$.  To this end, note that such a $\sigma$ satisfies $\sigma \cap \tau = \emptyset$ and $\sigma \cup \tau  \in \Delta(\mathcal{C}) \setminus \mathcal{C}$.  Thus, because $\mathcal{C}$ is locally good, 
the link 
${\rm Lk}_{\sigma \cup \tau} (\Delta(\mathcal{C})) = 
{\rm Lk}_{\sigma} ({\rm Lk}_{\tau} (\Delta(\mathcal{C} )))$ 
is contractible, and thus $\sigma$ is {\em not} a mandatory codeword of ${\rm Lk}_{\tau} (\Delta(\mathcal{C}))$.
  
We now apply Lemma~\ref{lem:code-cpx}:
the link ${\rm Lk}_{\tau} (\Delta(\mathcal{C}))$ is contractible (because $\mathcal{C}$ is locally good and $\tau \notin \mathcal{C}$), so
$ \mathcal{L}$ also is contractible.  Hence, $V_{\tau}$ is contractible.
\end{proof}

\begin{rem}
In the proof of Proposition~\ref{prop:realization},
the only place where we use the fact that $\mathcal C$ is  locally good, is at the end, when we show that $\mathcal L$ is contractible.
\end{rem}

\subsection{Main result}

Our next step is to modify, via a ``reverse deformation retract'', the realization from Proposition~\ref{prop:realization} so that the sets are open (Proposition~\ref{prop:realization-U}).  

To this end, for any $n$,
let $\Delta^n:=2^{[n]}$ be the complete simplicial complex on $[n]$, 
and consider an embedding of $|\Delta^n|$ in $\mathbb{R}^{n-1}$.
The facet-defining hyperplanes
of $|\Delta^n|$ form a hyperplane arrangement whose open chambers
(in the complement of the hyperplanes) 
correspond to nonempty subsets $\sigma$ of $[n]$ (see Figure~\ref{fig:R}).
Let $C_{\sigma}$ denote the closure of the chamber that corresponds to $\interior $.
Choose an ordering of 
the nonempty faces of $\Delta^n$ so that they are nondecreasing in dimension:
$0=\dim \sigma_1 \leq \dim \sigma_2 \leq \dots \leq \dim \sigma_{2^n-1}$.  
We define recursively the following sets (see the right-hand side of Figure~\ref{fig:R}):
\begin{align} \label{eq:R}
	R_{\sigma_k} ~:=~ 
	C_{\sigma_k} \setminus \left( \cup_{j=1}^{k-1} C_{\sigma_j} \right)
    ~.
\end{align}
We claim that the order of the $\sigma_i$'s does not matter.
Indeed, if $\dim \sigma_i = \dim \sigma_j$ (with $i \neq j$),
then the intersection $C_{\sigma_i} \cap C_{\sigma_j}$ is 
the (possibly empty) face $\sigma_i \cap \sigma_j$ of $|\Delta^n|$.
This face $\sigma_i \cap \sigma_j = \sigma_l$
is indexed by 
some $l < i, j$, and so this face,
regardless of whether $i <j$ or $j<i$, is neither in $R_{\sigma_i}$ nor $R_{\sigma_j}$.

We make the following several observations.  First, 
the $R_{\sigma}$'s are disjoint, and each $R_{\sigma}$ is convex and full-dimensional.  Also, $R_{\sigma}$ deformation retracts to $\interior $ via a deformation retract of $\mathbb{R}^{n-1}$ to $|\Delta^n|$.
Finally, the interior of $R_{\sigma}$ is the open chamber $C_{\sigma}$ of the hyperplane arrangement of $|\Delta^n|$'s facet-defining hyperplanes.

In analogy to how we built the sets $V_i$ from the sets $\interior$ (in 
Definition~\ref{def:code-cpx-realization}), we
now build sets $W_i$ from the sets $R_{\sigma}$ (which, we noted, deformation retract to the sets $\interior$).

\begin{prop} \label{prop:realization-W}
Let $\mathcal{C}$ be a code on $n$ neurons.
For each $i \in [n]$, consider the following subset of
$\mathbb{R}^{n-1}$:
	\begin{align} \label{eq:W}
		W_i ~:=~ \bigcup_{ i \in \sigma \in \mathcal{C}} R_{\sigma}~,
	\end{align}
where $R_{\sigma}$ is as in~\eqref{eq:R}. 
Then $\{W_1, W_2, \dots, W_n \}$ is a realization of $\mathcal{C} \setminus \{\emptyset\}$. 
\end{prop}

\begin{proof}
By construction, we have
	\begin{align*}
		\bigcup_{ i =1}^n W_i 
        ~=~ 
              \bigcup_{ \emptyset \neq \sigma \in \mathcal{C}}^{\cdot } R_{\sigma}~,
	\end{align*}
and the subset $R_{\sigma}$ realizes the codeword $\sigma$.
Define the stimulus space $X= \cup_{i=1}^n W_i$. 
Then the collection 
$\{W_1, W_2, \dots, W_n \}$ realizes $\mathcal{C} \setminus \{ \emptyset \}$ with respect to $X$.
\end{proof}

The sets $W_i$ in Proposition~\ref{prop:realization-W} are in general not open sets (see Figure~\ref{fig:R}), but they are unions of full-dimensional convex sets, so we now consider the interiors. 

\begin{prop} \label{prop:realization-U}
Let $\mathcal{C}$ be a code on $n$ neurons.
For each $i \in [n]$, consider the following subset of
$\mathbb{R}^{n-1}$:
	\begin{align*}
		U_i ~:=~ {\rm int} (W_i)~,
	\end{align*}
where $W_i$ is as in~\eqref{eq:W}.  
Then $\{U_1, U_2, \dots, U_n \}$ is a realization of $\mathcal{C}$.
Moreover, if $\mathcal{C}$ is locally good, then $\{U_1, U_2, \dots, U_n \}$ is a {\em good-cover} realization of $\mathcal{C} \setminus \{ \emptyset \}$.
\end{prop}

Before proving Proposition~\ref{prop:realization-U}, we give an example of the realization given in the proposition.

\begin{exa} \label{ex:realize-code-U}
We return to the locally good code $\mathcal{C} = \{123,12,23,1,2,\emptyset\}$
from Example~\ref{ex:realize-code}.
The good-cover realization $\{U_1,U_2,U_3\}$ of $\mathcal{C}$ from Proposition~\ref{prop:realization-U}, along with $U_1 \cup U_2 \cup U_3$,
is depicted here:

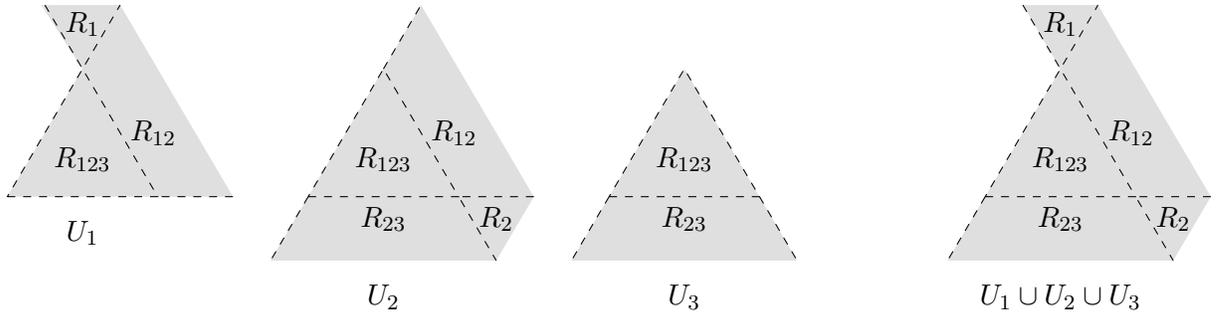
\begin{center}
\begin{tikzpicture}[scale=1]
    \tikzstyle{point}=[circle]
	    \node[label=below:{$U_1$}](u1) at (-8,-0.05){};
    \node[label=below:{$U_2$}](u1) at (-4,-0.9){};
    \node[label=below:{$U_3$}](u1) at (0,-0.9){};
    \node[label=below:{$U_1\cup U_2 \cup U_3$}](u1) at (5,-0.9){};
			    	\fill[gray!25] (-8.55,2.55)--(-8,1.7)--(-9,0)--(-6,0)--(-7.5,2.55);
	    \node (d1)[point] at (-9,0) {};
    \node (e1)[point] at (-7,0) {};
    \node (f1)[point] at (-8,1.7) {};
   	    \node (h1)[point] at (-9.5,-0.85) {};
 	    \node (i1)[point] at (-6,0) {};
    \node (j1)[point] at (-6.5,-.85) {};
	    \node (k1)[point] at (-8.5,2.55) {};
    \node (l1)[point] at (-7.5,2.55) {};
	    \draw[dashed]   (d1.center) -- (e1.center)node[pos=.5,above,yshift = .2cm]{$R_{123}$} -- (f1.center)node[pos=.5,right]{$R_{12}$} -- (d1.center);
	     \draw[dashed]   (i1.center) -- (e1.center);
	     \draw[dashed]   (k1.center) -- (f1.center) -- (l1.center)node[pos=0,above,yshift=.3cm]{$R_1$};			
     		    	\fill[gray!25] (-3.5,2.55)--(-5.5,-0.85)--(-2.5,-0.85)--(-2,0);
	    \node (d2)[point] at (-5,0) {};
    \node (e2)[point] at (-3,0) {};
    \node (f2)[point] at (-4,1.7) {};
   	    \node (h2)[point] at (-5.5,-0.85) {};
 	    \node (i2)[point] at (-2,0) {};
    \node (j2)[point] at (-2.5,-.85) {};
	    \node (k2)[point] at (-4.5,2.55) {};
    \node (l2)[point] at (-3.5,2.55) {};
	    \draw[dashed]   (d2.center) -- (e2.center)node[pos=.5,below]{$R_{23}$}node[pos=.5,above,yshift = .2cm]{$R_{123}$} -- (f2.center)node[pos=.5,right]{$R_{12}$} -- (d2.center);
     	     \draw[dashed]   (d2.center) -- (h2.center);
	     \draw[dashed]   (i2.center) -- (e2.center)node[pos=.5,below] {$R_2$} -- (j2.center);
	     \draw[dashed]   (f2.center) -- (l2.center);
     		    	\fill[gray!25] (0, 1.7)--(-1.5,-0.85)--(1.5,-0.85);
	    \node (d3)[point] at (-1,0) {};
    \node (e3)[point] at (1,0) {};
    \node (f3)[point] at (0,1.7) {};
   	    \node (h3)[point] at (-1.5,-0.85) {};
 	    \node (i3)[point] at (2,0) {};
    \node (j3)[point] at (1.5,-.85) {};
	    \node (k3)[point] at (-0.5,2.55) {};
    \node (l3)[point] at (0.5,2.55) {};
    	    \draw[dashed]   (d3.center) -- (e3.center)node[pos=.5,below]{$R_{23}$}node[pos=.5,above,yshift = .2cm]{$R_{123}$} -- (f3.center) -- (d3.center);
          \draw[dashed]   (d3.center) -- (h3.center);
	     \draw[dashed]   (e3.center) -- (j3.center);
     		    	\fill[gray!25] (4,0)--(6,0)--(5,1.7); 	\fill[gray!25] (4,0)--(3.5,-0.85)--(6.5,-0.85)--(6,0); 	\fill[gray!25] (6.5,-0.85)--(6,0)--(7,0); 	\fill[gray!25] (6,0)--(7,0)--(5.5,2.55)--(5,1.7); 	\fill[gray!25] (5.5,2.55)--(5,1.7)--(4.5,2.55); 	    \node (d)[point] at (4,0) {};
    \node (e)[point] at (6,0) {};
    \node (f)[point] at (5,1.7) {};
   	    \node (h)[point] at (3.5,-0.85) {};
 	    \node (i)[point] at (7,0) {};
    \node (j)[point] at (6.5,-.85) {};
	    \node (k)[point] at (4.5,2.55) {};
    \node (l)[point] at (5.5,2.55) {};
	    \draw[dashed]   (d.center) -- (e.center)node[pos=.5,below]{$R_{23}$}node[pos=.5,above,yshift = .2cm]{$R_{123}$} -- (f.center)node[pos=.5,right]{$R_{12}$} -- (d.center);
     	     \draw[dashed]   (d.center) -- (h.center);
	     \draw[dashed]   (i.center) -- (e.center)node[pos=.5,below] {$R_2$} -- (j.center);
	     \draw[dashed]   (k.center) -- (f.center) -- (l.center)node[pos=0,above,yshift=.3cm]{$R_1$};
\end{tikzpicture}
\end{center}

\end{exa}

\begin{proof}[Proof of Proposition~\ref{prop:realization-U}]
By construction, the $U_i$'s are open.
So, we must show that 
 (1) the $U_i$'s realize $\mathcal{C} \setminus \{ \emptyset \}$, that is, $\mathcal{C}(\{U_i\}) = \mathcal{C}  \setminus \{ \emptyset \}$, and 
 (2) if $\mathcal{C}$ is locally good, then 
	$\{U_1, U_2, \dots, U_n \}$ is a good cover. 

\underline{Proof of (1)}. 
We begin by proving the containment 
$\mathcal{C}(\{U_i\}) \supseteq \mathcal{C}  \setminus \{ \emptyset \}$, where we define the stimulus space to be $X= \cup_{i=1}^n U_i$ (so, $\emptyset \notin \mathcal{C}(\{U_i\})$). 
Take $\emptyset \neq \sigma \in \mathcal{C}$.  We must show that $\sigma \in \mathcal{C}(\{U_i\})$.  Let $x$ be in the interior of $R_{\sigma}$.  Then, by construction, $ x \in U_{\sigma} \setminus \left( \cup_{j \notin \sigma} U_j \right)$, so $\sigma \in \mathcal{C}(\{U_i\})$.

To prove the remaining containment, $\mathcal{C}(\{U_i\}) \subseteq \mathcal{C}  \setminus \{ \emptyset \}$,  
let $\sigma \in \mathcal{C}(\{U_i\})$.  
As explained above, $\sigma \neq \emptyset$.  Let $x \in U_{\sigma}\setminus \left( \cup_{j \notin \sigma} U_j \right)$.  
We consider two cases.
If $x$ is {\em not} on a facet-defining hyperplane of $|\Delta^n|$, then $x$ is in the interior of some $R_{\tau}$.
It is straightforward to check that $\tau = \sigma$, so ${\rm int}(R_{\sigma}) \subseteq U_{\sigma}$, and thus $\sigma \in \mathcal{C}$.

In the second case,
$x$ is on exactly $m$ facet-defining hyperplanes of $|\Delta^n|$ (for some $m \geq 1$).
Crossing exactly one such hyperplane means going from some region $R_{\alpha}$ to some region $R_{\alpha \cup \{i\}}$ (for some $i \notin \alpha$), or vice-versa.  Thus,  a small neighborhood $B$ of $x$ intersects exactly $2^m$ regions $R_{\tau}$: precisely those with 
\begin{align} \label{eq:tau}
	\eta \subseteq \tau \subseteq \widetilde \eta~,
\end{align}
for some $\eta \subseteq \widetilde \eta \subseteq [n]$ with $|\eta|+m=|\widetilde \eta|$.

Recall that $x$ is in the open set $U_{\sigma}$, so 
$B \subseteq U_{\sigma}$, and thus the interiors of the $2^m$ regions $R_{\tau}$ are contained in $U_{\sigma}$.  So, all sets $\tau$ given in~\eqref{eq:tau} are in the code $\mathcal{C}$.  Thus, to show that $\sigma \in \mathcal{C}$, it suffices to show that $\eta = \sigma$.
 The containment $\eta \supseteq \sigma$ follows from the fact that ${\rm int}(R_{\eta}) \subseteq U_{\sigma}$ (as explained above).
 We prove $\eta \subseteq \sigma$ by contradiction.  
Assume that there exists $k \in \eta \setminus \sigma$.
Then we have: 
\begin{align*}
	x ~\in~ B 
    ~\subseteq~ \cup_{\eta \subseteq \tau \subseteq \widetilde \eta~} R_{\tau} 
    ~\subseteq ~U_{\eta} ~\subseteq ~U_k~.
\end{align*}
Thus, $x \in U_k$, where $k \notin \sigma$, which  contradicts the choice of $x$.  So, $\eta \subseteq \sigma$ holds. 

\underline{Proof of (2)}. 
Assume that $\mathcal{C}$ is locally good.
 Let $\emptyset \neq \tau \subseteq [n]$, and assume that $U_{\tau}$ is nonempty.  We must show that 
 $U_{\tau}$ is contractible.  
 By Proposition~\ref{prop:realization}, $V_{\tau}$ is contractible, so it is enough to show that 
 $U_{\tau} \simeq V_{\tau}$.  
 We will prove this by showing that the set $Z_{\tau} := U_{\tau} \bigcup V_{\tau}$ deformation retracts to $V_{\tau}$ and also weak deformation retracts~\cite{hatcher}  to $U_{\tau}$.  
 It is straightforward to check that $U_{\tau} = {\rm int}(W_{\tau})$, and $W_{\tau}$ is a (full-dimensional) set that deformation retracts to $V_{\tau}$.  
 Thus, $Z_{\tau}=U_{\tau} \bigcup V_{\tau}$ deformation retracts to $V_{\tau}$.
Finally, we obtain a weak deformation retract of $Z_{\tau}$ to $U_{\tau}$
via the following homotopy: we simultaneously translate each facet-defining hyperplane of the simplex $| \Delta^n|$ a small distance away from the simplex, so that the subset of $Z_{\tau}$ that is ``swept up'' is pushed into $U_{\tau}$.   
 \end{proof}

It is natural to ask whether a closed-set version of Proposition~\ref{prop:realization-U}
holds:
	\begin{que} \label{q:closed}
	Is every locally good code a closed-good-cover code?
	\end{que}
\noindent
We could try to resolve this question by 
``closing'' our open-good-cover realizations; we ask, For a locally good code~$\mathcal{C}$, is $\{\overline{U_i}\}$ a closed-good-cover realization
of $\mathcal{C}$, where $\{U_i\}$ is as in Proposition~\ref{prop:realization-U}?
However, this is not true.
Indeed, 
it is easy to check that for the locally good code 
$\mathcal{C}= \{1,12,13\}$, 
the open realization 
$\{U_1, U_2, U_3\}$,
when we take closures, 
yields a code $\mathcal{C}( \{ \overline{U_1},
\overline{U_2},\overline{U_3} \} )$ with the
``extra'' codeword 123.

\begin{thm} \label{thm:iff}
A code is locally good if and only if it is a good-cover code.
\end{thm}
\begin{proof}
The backward direction is in Proposition~\ref{prop:locally-good}. The forward direction follows from Proposition~\ref{prop:realization-U} and 
the fact that if $\mathcal{C} \cup \{\emptyset\}$ is a good-cover code, then so is $\mathcal{C}$ (Remark~\ref{rmk:empty-set}).
\end{proof}

Finally, note that the good-cover realizations from Proposition~\ref{prop:realization-U} 
are embedded in an $(n-1)$-dimensional Euclidean space, where $n$ is the number of neurons. 
The following result shows that this embedding dimension can not, in general, be improved.
\begin{prop} \label{prop:emb-dim}
The code on $n$ neurons comprising all proper subsets of $[n]$, 
\[
\mathcal C_n ~:=~
	2^{[n]} \setminus \{123\dots n\}~,
\]
is locally good, and can not be realized by a good cover in $\mathbb{R}^{n-2}$.
\end{prop}
\begin{proof}
The code $\mathcal C_n$ is locally good, as it is a simplicial complex and hence contains all its mandatory codewords.
Next, suppose for contradiction that $\mathcal U=\{U_1, U_2, \dots, U_{n}\}$ forms a good-cover realization of~$\mathcal C_n$ in $\mathbb R^{n-2}$.  By the nerve lemma (Proposition~ \ref{prop:nerve-thm}), the union $\cup_{i=1}^n U_i$ is homotopy equivalent to the nerve of $\mathcal U$, which is the boundary of the $(n-1)$-simplex.  However, no subset of $\mathbb R^{n-2}$ is homotopy equivalent to an $(n-2)$-sphere.  We have reached a contradiction.
\end{proof}

\section{Undecidability of the good-cover decision problem}
\label{sec:undecidability}
Having shown that being locally good is equivalent to being a good-cover code (Theorem~\ref{thm:iff}),
we now prove that the corresponding decision problem is undecidable (Theorem~\ref{thm:undecidable}).
The proof hinges on the undecidability of determining whether a homology ball 
is contractible (Lemma~\ref{lem:undecidable})~\cite{tancer-complexity}.

We say that a code $\mathcal{C}$ is $k$-sparse if its simplicial complex $\Delta(\mathcal{C})$ has dimension at most $k-1$:
\begin{defn}
A code $\mathcal{C}$ is \textbf{$k$-sparse} if $|\sigma| \leq k$ for all $\sigma \in \mathcal{C}$.
\end{defn}

\begin{lem}[Tancer~\cite{tancer-complexity}] \label{lem:undecidable}
The problem of deciding whether a given 4-dimensional simplicial complex is contractible is undecidable.
\end{lem}

\begin{thm}[Undecidability of the good-cover decision problem] \label{thm:undecidable}
The problem of deciding whether a 5-sparse code is locally good (or, equivalently, has a good cover) is undecidable.
\end{thm}
\begin{proof}
Given any 4-dimensional simplicial complex $\Delta$, consider the cone over $\Delta$ on a new vertex ${v}$. This cone is itself a simplicial complex, which we denote by $\Delta'$. 
Let $\mathcal{C}$ denote the neural code $\Delta' \setminus \{ {v} \}$, which is 5-sparse. 
Note that $\Delta(\mathcal{C})=\Delta'$, so the only codeword in $\Delta(\mathcal{C})$ that is missing from $\mathcal{C}$ is ${v}$.  So, by Proposition~\ref{prop:mandatory}, the code $\mathcal{C}$ is locally good if and only if ${\rm Lk}_{{v}}(\Delta')=\Delta$ is contractible. Thus, any algorithm that could decide whether $\mathcal{C}$ is locally good would also decide whether $\Delta$ is contractible, which is impossible by Lemma~\ref{lem:undecidable}.
\end{proof}

\subsection{Decidability for 3-sparse and 4-sparse codes}

Can the condition of 5-sparsity in Theorem~\ref{thm:undecidable} can be extended to 4-sparsity or even 3-sparsity?  For 3-sparsity, the answer is ``no'':

\begin{prop} \label{prop:3-sparse}
The problem of determining whether a 3-sparse code is locally good (or, equivalently, has a good cover) is decidable.
\end{prop}

\begin{proof}
It is straightforward to write an algorithm that, for a given code $\mathcal{C}$, enumerates the nonempty intersections of maximal codewords.
By Proposition~\ref{prop:mandatory}, $\mathcal{C}$ is locally good if and only if the links of these intersections are all contractible.
Hence, we are interested in the decision problem for determining whether these links are contractible.  
When $\mathcal{C}$ is 3-sparse, these links are 2-sparse, i.e., (undirected) graphs.  Contractible graphs are precisely trees (connected graphs without cycles), and the problem of determining whether a graph is a tree is decidable.
\end{proof}

As for extending Theorem~\ref{thm:undecidable} to 4-sparsity, this problem is open.  This is because it is unknown whether the dimension in 
Lemma~\ref{lem:undecidable} can be lowered from 4 to 3~\cite[Appendix~A]{tancer-complexity}.

\subsection{Relation to the convexity decision problem}
For 2-sparse codes, being convex is equivalent to being locally good~\cite{sparse}. So, by Proposition~\ref{prop:3-sparse}, the convexity decision problem for these codes is decidable.

For codes without restriction on the sparsity, 
we revisit, from the proof of Theorem~\ref{thm:undecidable}, codes of the form $\mathcal{C}=\Delta' \setminus \{ {v} \}$, where $\Delta'$ is a cone over a simplicial complex $\Delta$ on a new vertex {v}.  Consider the case when $\Delta$ is contractible.  
Does it follow that $\mathcal{C}$ is convex?  
If it were, then by an argument analogous to the proof of Theorem~\ref{thm:undecidable}, the convexity decision problem would be undecidable.  However, 
we will see in the next section that 
there exist such codes $\mathcal{C}$ that are non-convex (Example~\ref{ex:good-cover-does-not-imply-loc-great}).
Indeed, the convexity decision problem is unresolved, so we pose it here.
\begin{que} \label{q:cvx-decidable}
Is the problem of determining whether a code is convex, decidable?
\end{que}

In the next section, we will introduce a superset of all convex codes, the ``locally great'' codes, and show that the corresponding decision problem is decidable (Theorem~\ref{thm:decidable}).

\section{A new, stronger local obstruction to convexity} \label{sec:new-obs}

Recall that a code $\mathcal{C}$ has no local obstructions if and only if it contains all its mandatory codewords (Proposition \ref{prop:mandatory}), and these mandatory codewords are precisely the faces of the simplicial complex $\Delta(\mathcal{C})$ whose link is non-contractible.  We prove in this section that by replacing ``non-contractible'' by ``non-collapsible'' (Definition~\ref{def:d-collapse}) we obtain a stronger type of local obstruction to convexity (Theorem~\ref{thm:new-obstruction}).

This result yields a new family of codes that, like the earlier counterexample code (Proposition~\ref{prop:counterexample-code}), are locally good, but not convex.  
This new family comprises codes of the form $\mathcal{C}= \widetilde \Delta \setminus \{v\}$, where $\widetilde \Delta$ is a cone, on a new vertex $v$, over a contractible but non-collapsible simplicial complex $\Delta$ (see Example~\ref{ex:good-cover-does-not-imply-loc-great}).  
Such a code $\mathcal{C}$ (as we saw in the proof of Theorem~\ref{thm:undecidable}) is missing only one codeword (namely, $v$) from its simplicial complex  
(because $\Delta(\mathcal{C})=\widetilde \Delta$).
We will use this several times in this section.

\subsection{Background on collapses} \label{sec:collapsibility}
First we make note of a somewhat overloaded definition in the literature. The notion of a $\emph{collapse}$ of a simplicial complex was introduced by Whitehead in 1938~\cite{whitehead}. The more general concept of $d$-\emph{collapse} was introduced by Wegner in 1975~\cite{wegner-collapse}.

\begin{defn} \label{def:d-collapse}
Let $\Delta$ be a simplicial complex, and let $\mathcal{M}$ be the set of its facets. 
\begin{enumerate}
\item For any face $\sigma$ of $\Delta$ such that there is a unique $\tau \in \mathcal{M}$ for which $\sigma \subseteq \tau$,
we define
$$
\Delta' = \Delta \setminus \{\nu \in \Delta \mid \sigma \subseteq \nu\}~,
$$ 
and say that $\Delta'$ is an \textbf{elementary $\textbf{d}$-collapse} of $\Delta$ induced by $\sigma$ (or $(\sigma,\tau)$). 
This elementary ${d}$-collapse is denoted by $\Delta \to \Delta'$. 
(Here $d$ refers to the constraint dim $\sigma \leq d - 1$, but
in this work we will let $d$ be arbitrarily high when we use the term ``$d$-collapse''). A sequence of elementary $d$-collapses 
$$
\Delta \to \Delta_1 \to \Delta_2 \to \dots \to \Delta_n
$$
is a \textbf{$\textbf{d}$-collapse} of $\Delta$ to $\Delta_n$. 
\item An  \textbf{elementary collapse} is an elementary $d$-collapse induced by a face $\sigma$
that is {\em not} a facet (i.e., $\sigma \subsetneq \tau$). A sequence of elementary collapses
starting with $\Delta$ and ending with $\Delta'$
is a \textbf{collapse} of $\Delta$ to $\Delta'$. Finally, a simplicial complex is \textbf{collapsible} if it collapses to a point (via some sequence).  
\end{enumerate}
\end{defn}

\begin{rem} \label{rmk:collaps-2-defns}
Following~\cite{tancer-complexity}, our 
Definition~\ref{def:d-collapse}(2)
characterizes ``collapsible'' in terms of
elementary collapses induced by pairs $(\sigma,\tau)$ with $\dim \sigma < \dim \tau$.
An equivalent definition, also used in the literature~\cite{bjorner,cohen},
is via elementary collapses under a stronger condition: $\dim \sigma = \dim \tau - 1$.
For completeness, we prove in Appendix~B that these two definitions of collapsible are equivalent
(Proposition~\ref{prop:2-collaps-defns-equivalent}).
\end{rem}

\begin{exa}[A $d$-collapse] \label{ex:d-collapse}
The following is an elementary $d$-collapse defined by $\sigma=\tau = 123$:
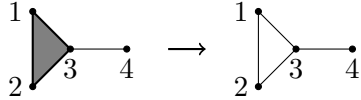
\begin{center}
	\begin{tikzpicture}[scale=.5]
        \draw [fill=gray, thick] (0,0) --(1,1) --(0,2) -- (0,0);
    	\draw (2.5,1) --(1,1);
        \draw [fill] (0,0) circle [radius=0.08];
    \draw [fill] (1,1) circle [radius=0.08];
    \draw [fill] (0,2) circle [radius=0.08];
    \draw [fill] (2.5,1) circle [radius=0.08];
        \node [left] at (0,2) {$1$};
    \node [left] at (0,0) {$2$};
    \node [below] at (1,1) {$3$};
    \node [below] at (2.5,1) {$4$};
	    \draw [->,thick] (3.6,1) -- (4.6,1);
	       \draw [fill] (6,0) circle [radius=0.08];
    \draw [fill] (7,1) circle [radius=0.08];
    \draw [fill] (6,2) circle [radius=0.08];
    \draw [fill] (8.5,1) circle [radius=0.08];
        \node [left] at (6,2) {$1$};
    \node [left] at (6,0) {$2$};
    \node [below] at (7,1) {$3$};
    \node [below] at (8.5,1) {$4$};
    	\draw (8.5,1) --(7,1);
	\draw (6,0) --(7,1);
	\draw (6,2) --(7,1);
	\draw (6,0) --(6,2);
    \end{tikzpicture}
\end{center}
\end{exa}
\begin{exa}[A collapse] \label{ex:collapse}
The following is a collapse to a point:
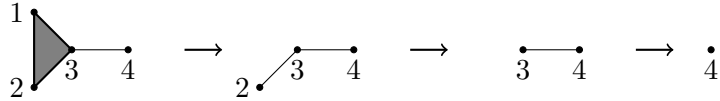
\begin{center}
	\begin{tikzpicture}[scale=.5]
        \draw [fill=gray, thick] (0,0) --(1,1) --(0,2) -- (0,0);
    	\draw (2.5,1) --(1,1);
        \draw [fill] (0,0) circle [radius=0.08];
    \draw [fill] (1,1) circle [radius=0.08];
    \draw [fill] (0,2) circle [radius=0.08];
    \draw [fill] (2.5,1) circle [radius=0.08];
        \node [left] at (0,2) {$1$};
    \node [left] at (0,0) {$2$};
    \node [below] at (1,1) {$3$};
    \node [below] at (2.5,1) {$4$};
	    \draw [->,thick] (4,1) -- (5,1);
	    	\draw (8.5,1) --(7,1);
	\draw (6,0) --(7,1);
        \draw [fill] (6,0) circle [radius=0.08];
    \draw [fill] (7,1) circle [radius=0.08];
    \draw [fill] (8.5,1) circle [radius=0.08];
        \node [left] at (6,0) {$2$};
    \node [below] at (7,1) {$3$};
    \node [below] at (8.5,1) {$4$};
	    \draw [->,thick] (10,1) -- (11,1);
	    	\draw (13,1)--(14.5,1);
        \draw [fill] (13,1) circle [radius=0.08];
    \draw [fill] (14.5,1) circle [radius=0.08];
        \node [below] at (13,1) {$3$};
    \node [below] at (14.5,1) {$4$};
	    \draw [->,thick] (16,1) -- (17,1);
	        \draw [fill] (18,1) circle [radius=0.08];
        \node [below] at (18,1) {$4$};
	\end{tikzpicture}
\end{center}
This collapse arises from the elementary collapses 
induced by, respectively,
$(\sigma,\tau)=(1,123)$, 
$(\sigma,\tau)=(2,23)$, and
$(\sigma,\tau)=(3,34)$.
\end{exa}

Comparing Examples~\ref{ex:d-collapse} and~\ref{ex:collapse}, note that the homotopy type was {\em not} preserved throughout the $d$-collapse, 
but was preserved through the collapse.  This is explained in the next two results, the first of which is well known (see, e.g.,~\cite[Ch.\ 1.2]{cohen} or \cite{bjorner}). 

\begin{lem} \label{lem:coll-implies-contractible} 
Elementary collapses preserve homotopy type. Thus, 
collapsible implies contractible.
\end{lem}
However, not every contractible simplicial complex is collapsible. 
One example is any triangulation of the 2-dimensional topological space known as Bing's house with two rooms~\cite[Ex.\ 1.6.1]{rushing}.
Another example is the dunce hat~\cite{dunce}.

Lemma~\ref{lem:coll-implies-contractible} extends as follows:
\begin{lem} \label{lem:d-collapse}
Let $\Delta$ be a contractible simplicial complex, and let $\Delta \to \Delta'$ be an elementary $d$-collapse induced by a face $\sigma$.  
Assume that $\Delta'$ contains a nonempty face.
Then $\Delta'$ is contractible if and only if $\sigma$ is {\em not} a facet (i.e., $\Delta \to \Delta'$ is an elementary collapse).
\end{lem}
\begin{proof} 
The backward direction is immediate from Lemma~\ref{lem:coll-implies-contractible}.
We prove the contrapositive of the forward direction.  Suppose that $\sigma$ is a facet of $\Delta$. Then $|\Delta| = [\sigma] \cup |\Delta'|$, where
$[\sigma]$ denotes the the topological realization of $\sigma$ as a facet of $|\Delta|$.
With an eye toward using the Mayer-Vietoris sequence,
we note that the intersection $[\sigma] \cap |\Delta'|$ is the boundary of the simplex $[\sigma]$ and thus is a sphere and so has a non-vanishing homology group.
Applying the Mayer-Vietoris sequence for homology to $|\Delta| = [\sigma] \cup |\Delta'|$, and using the hypothesis that $\Delta$ is contractible, we conclude that $|\Delta'|$ has a non-vanishing homology group, and therefore is not contractible. 
\end{proof}

Finally, we state the following result of Wegner~\cite{wegner-collapse}; see also Tancer's description in~\cite[\S2]{tancer-survey}.  

\begin{lem}[Wegner~\cite{wegner-collapse}] \label{lem:sweep}
Let $\mathcal{W}=\{ W_1, W_2, \dots, W_n\}$ be a collection of nonempty convex (not necessarily open) sets in $\mathbb{R}^m$. 
Let $\Lambda$ be the nerve of $\mathcal{W}$.
Then there exists an open halfspace $H$ in $\mathbb{R}^m$ such that, 
letting $\Lambda'$ denote the nerve of the collection $\{W_1 \cap H, W_2 \cap H, \dots, W_n \cap H\}$,
the following is an elementary $d$-collapse:
$\Lambda \to \Lambda'$. \end{lem}

Informally speaking, Wegner proved Lemma~\ref{lem:sweep} by 
sweeping a hyperplane from infinity across $\mathbb{R}^m$, deleting everything in its path,
until an intersection region corresponding to a facet $\tau$ of $\Lambda$ has been removed. 
This yields the elementary $d$-collapse $\Lambda \to \Lambda'$ induced by some face contained in $\tau$.

\begin{rem} \label{rmk:collapse-polar}
Recent work of Itskov, Kunin, and Rosen is similar in spirit to ours: 
they show that the collapsibility of a certain simplicial complex associated to a code (the ``polar complex'')
ensures that the code avoids certain obstructions to being a certain type of convex code (namely, a code arising from a ``non-degenerate'' hyperplane arrangement)~\cite[\S 6.5]{IKR}.
\end{rem}

\subsection{A key lemma} \label{sec:key-lemma}
This subsection contains the key result (Lemma~\ref{lem:key}) 
that allows us to 
establish, via Theorem~\ref{thm:new-obstruction}, 
our new local obstruction.
Lemma~\ref{lem:key}
states that for an open cover by convex sets of a set that is itself convex, the corresponding nerve is collapsible.
The original local obstruction (Proposition~\ref{prop:locally-good})
relied on a weaker version of this result, which states only that such a nerve is contractible.
Accordingly, 
the way we use Lemma~\ref{lem:key} to prove
Theorem~\ref{thm:new-obstruction} 
is analogous to how the authors of~\cite{no-go,what-makes} used 
the ``contractible'' version of the lemma to establish the original notion of local obstruction.

\begin{lem} \label{lem:key}
Let $\mathcal{W}=\{ W_1, W_2, \dots, W_n\}$ be a collection of convex open sets in $\mathbb{R}^m$ such that 
their union $W_1 \cup W_2 \cup \dots \cup W_n$ is nonempty and convex.
Then the nerve of $\mathcal{W}$ is collapsible.
\end{lem}
\begin{proof}
Let $\Lambda$ denote the nerve of $\mathcal{W}$.  
Let $p$ denote the number of nonempty faces of $\Lambda$.  Note that $p \geq 1$, as the union of the $W_i$'s is nonempty.
We proceed by induction on $p$.

Base case: $p=1$. Then $\Lambda$ is a point, and thus is collapsible.

Inductive step: $p\geq 2$.  
Assume that the lemma is true for all nerves $\Lambda$ with at most $p-1$ nonempty faces.  
First consider the case when $\Lambda$ has only one facet.  Then $\Lambda$ is a simplex, and every simplex is collapsible.

Now consider the remaining case, when $\Lambda$ has at least two facets. 
Without loss of generality, each $W_i$ is nonempty (deleting $W_i$'s that are empty does not affect the union or the nerve).  So, by Lemma~\ref{lem:sweep}, there exists an open halfspace $H$ such that $\Lambda \to \Lambda'$ is an elementary $d$-collapse, where $\Lambda'$ is the nerve of $\mathcal{V}:= \{  W_1 \cap H, W_2 \cap H, \dots, W_n \cap H  \}$.  We see that $\mathcal{V}$ is a collection of convex open sets whose union, $(\cup_{i \in [n]} W_i) \cap H$, is convex (because both $\cup_{i \in [n]} W_i$ and $H$ are convex)
and nonempty (indeed, the nerve $\Lambda'$ is nonempty, because $\Lambda$ has at least two facets and so at least one was unaffected by the elementary $d$-collapse $\Lambda \to \Lambda'$).  Thus, by the induction hypothesis, $\Lambda'$ is collapsible.

Thus, by definition of collapsible, we need only show that the elementary $d$-collapse $\Lambda \to \Lambda'$ was in fact an elementary collapse.  
To see this, note that the nerve theorem (Proposition~\ref{prop:nerve-thm}) implies that the nerve $\Lambda'$ is homotopy equivalent to $(\cup_{i \in [n]} W_i) \cap H$, which we saw above is convex (and nonempty) and thus contractible.  
Hence, by Lemma~\ref{lem:d-collapse}, $\Lambda \to \Lambda'$ is an elementary collapse.  This completes the proof.
\end{proof}

\subsection{Locally great codes} \label{sec:new-obstruction-subsection}

The following result gives a new class of local obstructions.

\bthm \label{thm:new-obstruction} 
Let $\mathcal{C}$ be a convex code.
Then for any  $\sigma \in \Delta(\mathcal{C}) \setminus \mathcal{C}$, the link $\Lk_\sigma(\Delta(\mathcal{C}))$ is collapsible.
\ethm

\begin{proof}
Let $U_1, U_2, \dots, U_n$ be convex open sets 
(in some stimulus space $X \subseteq \mathbb{R}^m$) that realize $\mathcal{C}$. 
Let $\sigma \in \Delta(\mathcal{C})\setminus \mathcal{C}$.  
Then, by definition,  
$\emptyset \neq U_{\sigma} \subseteq \cup_{i \notin \sigma} U_i$.  
Thus, 
$\mathcal{W} := \{ U_i \cap U_{\sigma} \}_{i \notin \sigma}$
is a collection of convex open sets such that 
their union equals $U_{\sigma}$ 
(and thus this union is nonempty and convex).
So, by Lemma~\ref{lem:key}, the nerve of $\mathcal{W}$ is collapsible.  It is straightforward to check 
that this nerve equals $\Lk_\sigma(\Delta(\mathcal{C}))$.  
So, as desired, the link is collapsible.
\end{proof}

\begin{defn} \label{def:locally-great}
A code $\mathcal{C}$ has a 
\textbf{local obstruction of the second kind}
if there exists $\sigma \in \Delta(\mathcal{C}) \sm \mathcal{C}$
such that the link  $\Lk_\sigma(\Delta(\mathcal{C}))$ is {\em not} collapsible.
If $\mathcal{C}$ has {\em no} local obstructions of the second kind, 
$\mathcal{C}$ is \textbf{locally great}.
\end{defn}

The next result follows from Theorem~\ref{thm:new-obstruction} and the fact that collapsible implies contractible (Lemma~\ref{lem:coll-implies-contractible}). 
\begin{cor} \label{cor:new-implications} 
$\mathcal{C}$ is convex 
$\quad \Rightarrow \quad$
$\mathcal{C}$ is locally great 
$\quad \Rightarrow \quad$
$\mathcal{C}$ is locally good. \end{cor}

Neither implication in Corollary~\ref{cor:new-implications} is an equivalence, as we see in the following examples.

\begin{exa}[Locally great does {\em not} imply convex] \label{ex:loc-great-does-not-imply-cvx}
The counterexample code from Proposition~\ref{prop:counterexample-code}
is non-convex~\cite{LSW}, but, we claim,
locally great.
To verify this, we must check that for every missing codeword
$\sigma \in \Delta(\mathcal{C}) \setminus \mathcal{C} = \{234,235, 245, 345, 12, 15, 24, 25, 35, 1, 2, 5 \}$,
the corresponding link ${\rm Lk}_{\sigma}(\Delta(\mathcal{C}))$ is collapsible.  We accomplish this as follows:
\begin{itemize}
\item When 
$\sigma \in \{234,235, 245, 345, 12, 15 \}$,
the link is a point and thus collapsible.
\item When 
$\sigma \in \{ 24, 25, 35, 1 \}$,
the link is a single edge and thus collapsible (cf.\ Example~\ref{ex:collapse}). 
\item When 
$\sigma \in \{2, 5 \}$,
the link is, up to relabeling, the simplicial complex in Example~\ref{ex:collapse}, which we showed is collapsible.
\end{itemize}

\end{exa}

\begin{exa}[Locally good does {\em not} imply locally great] \label{ex:good-cover-does-not-imply-loc-great}
Let $\Delta$ be a triangulation of Bing's house, 
so $\Delta$ is a 2-dimensional simplicial complex that is contractible but not collapsible~\cite[Ex.\ 1.6.1]{rushing}.
Let $\widetilde \Delta$ be a cone over $\Delta$ on a new vertex $v$, and consider the code
$\mathcal{C}:= \widetilde \Delta \setminus \{v\}$.

We claim that $\mathcal{C}$ is locally good,
but not locally great.
Indeed, the only ``missing'' 
codeword of $\mathcal{C}$ is $v$, and its link is ${\rm Lk}_v(\Delta(\mathcal{C})) = \Delta$, which is contractible but not collapsible.   Thus, $\mathcal{C}$ is locally good (by Proposition~\ref{prop:mandatory}), but not locally great (by definition).
\end{exa}

We now consider the question of whether there might be an even stronger local obstruction beyond ``locally great'' (perhaps ``locally excellent''?).  That is, can the conclusion of Lemma~\ref{lem:key}
be strengthened from ``collapsible'' to some more general property of simplicial complexes?  Or, on the contrary, does the converse of Lemma~\ref{lem:key} hold?  Accordingly, we pose the following question:
\begin{que}\label{q:collaps-nerve}
Is every collapsible simplicial complex the nerve of a convex open cover of some \underline{convex} set?
\end{que}
\noindent  
After our manuscript appeared on arXiv, Question~\ref{q:collaps-nerve} was resolved in the negative
 by Jeffs and Novik~\cite{jeffs2018}, who also proved further results about neural codes that can be realized by convex open sets whose union is convex.

\begin{rem} \label{rmk:decision}
Some decision problems related to Question~\ref{q:collaps-nerve} have been resolved.
The problem of deciding whether a simplicial complex is the nerve of a convex open cover of some (not necessarily convex) subset of $\mathbb{R}^d$ is decidable~\cite{wegner,tancer-survey}. 
On the other hand, the 
problem of deciding whether a simplicial complex is the nerve of a good cover of some subset of $\mathbb{R}^d$, when $d \geq 5$, is undecidable~\cite{t-t}.
\end{rem}

\subsection{The locally-great decision problem}

In this subsection, we prove that the locally-great decision problem is decidable.  This result relies on the following lemma.

\begin{lem}[Tancer~\cite{tancer-complexity}] \label{lem:decidable}
The problem of deciding whether a given simplicial complex is collapsible is NP-complete.
\end{lem}

\begin{thm}[Decidability of the locally-great decision problem] \label{thm:decidable}
The problem of deciding whether a neural code is locally great is NP-hard.
\end{thm}
\begin{proof}
By Definition~\ref{def:locally-great} and Lemma~\ref{lem:decidable}, the following steps form an algorithm that determines whether a code is locally great: enumerate $\Delta(\mathcal{C}) \setminus \mathcal{C}$ and the corresponding links, and then check whether any of these links is collapsible.

To show this is NP-hard, it suffices to reduce (in polynomial time) the problem of deciding whether a simplicial complex is collapsible (which is NP-complete by Lemma~\ref{lem:decidable}) to the locally-great decision problem.  To this end, we proceed as in the proof of Theorem~\ref{thm:undecidable}:
given any simplicial complex $\Delta$, consider the cone over $\Delta$ on a new vertex ${v}$. 
This cone is itself a simplicial complex, which we denote by $\Delta'$. 
Let $\mathcal{C}=\Delta' \setminus \{ {v} \}$. 
Then, ${v}$ is the only codeword in $\Delta(\mathcal{C})$ that is missing from $\mathcal{C}$.  So, by Definition~\ref{def:locally-great}, the original simplicial complex 
$\Delta= {\rm Lk}_{{v}}(\Delta')$ is collapsible
if and only if 
the code $\mathcal{C}$ is locally great.
\end{proof}

\section{Discussion} \label{sec:discussion}
We return to the question that opened this work, Which neural codes are convex?  
There is a growing literature tackling this question, and here we resolved some foundational problems in this developing theory.
In summary, we now know the following:
\begin{center}
$\mathcal{C}$ is convex 
$~ \xRightarrow{(a)} ~$
$\mathcal{C}$ is locally great 
$~ \xRightarrow{(b)} ~$
$\mathcal{C}$ is a good-cover code 
$ ~\Leftrightarrow~ $
$\mathcal{C}$ is locally good 
$ ~\xRightarrow{(c)} ~$
$\mathcal{C}$ is connected, 
\end{center}
and the implications {\em (a)}--{\em (c)} are not equivalences.
Also, it is undecidable to tell whether an arbitrary neural code $\mathcal{C}$ is locally good, NP-hard to tell whether $\mathcal{C}$ is locally great, and the problem remains open for determining whether $\mathcal{C}$ is convex (Question~\ref{q:cvx-decidable}).

An additional problem suggested by our work pertains to detecting our new local obstructions (those of the second kind).  Neural codes have been studied from an algebraic standpoint via neural ideals, which are closely related to the Stanley-Reisner ideal of a code's simplicial complex~\cite{neural_ring,what-makes,GB,polarization,HMO,neural-ideal-homom,morvant}. Using these algebraic tools, it is possible to find local obstructions that can be detected by homology, which suffices to determine contractibility for small simplicial complexes. It 
would be interesting to see whether collapsibility of links can be characterized in a similar way, as unlike contractibility the former is decidable.  

Next, we revisit the implication, convex $\Rightarrow$ locally good.  The converse is false in general, with the first counterexample a 4-sparse code~\cite{LSW}, but is true for $2$-sparse codes~\cite{sparse}.  We ask, therefore, {\em Is every $3$-sparse locally good code, convex?}
Indeed, for $\mathcal{C}$ a 3-sparse code, mandatory codewords have size one or two, so the corresponding links are graphs, where contractibility and collapsibility are equivalent (and thus so are the concepts of locally great and locally good).
This problem therefore seems tractable.

Finally, it would be interesting to investigate how the theory of convex codes changes when considering closed 
rather than open realizations.  
Such a theory was initiated by Cruz {\em et al.}, who proved that there are open-convex codes that are not closed-convex and vice-versa~\cite{giusti-preprint}.  
We also know that every closed-convex code is a closed-good-cover code (by definition), while the converse is false~\cite{giusti-preprint}.  
Every closed-convex code also is locally great (Lemma~\ref{lem:key} and 
Theorem~\ref{thm:new-obstruction} easily generalize to closed sets, as a version of the nerve theorem applies to closed sets in $\mathbb{R}^d$). Also, every code that is a closed-good-cover code 
is locally good (by the nerve theorem -- see ~\cite{giusti-preprint}).
A related open question posed earlier is 
whether 
every locally good code
is a closed-good-cover code (Question~\ref{q:closed}).  
If so, then every locally good code also would be a closed-good-cover code, unifying some of the theory of open-convex and closed-convex codes.

We end by revisiting the neuroscience motivation for our work.  
Recall that place cells enable the brain to represent an organism's environment
by way of neural codes arising from approximately convex receptive fields.  
The goal of analyzing which codes are convex, therefore, is 
aimed at understanding what types of codes allow the brain to represent structured environments.  
One of the contributions here is to show that while the convex-code decision problem is still open, some related decision problems are hard.
Analyzing data from many neurons, therefore, may be computationally challenging.  On the other hand, our work reveals the meaning behind the main tool used for precluding convexity, namely, local obstructions: these obstructions exactly characterize when a code can {\em not} be realized by an (open) good cover.

\subsection*{Acknowledgements} 
This work was initiated by AC at the 2016 REU 
in the Department of Mathematics
at Texas A\&M University,
which was supported by the NSF (DMS-1460766). 
The authors thank
Chad Giusti, 
Laura Matusevich, 
Kaitlyn Phillipson, Zvi Rosen, 
Ola Sobieska, 
Martin Tancer, 
and Zev Woodstock 
for insightful comments and discussion.  
The authors also thank two conscientious referees whose comments improved this work.
AS was partially supported by the NSF (DMS-1312473/DMS-1513364)
and the Simons Foundation (\#521874).

\bibliography{convexity-refs}{}
\bibliographystyle{plain}

\appendix

\section{Removing interiors of faces} 
Here we prove Lemma~\ref{lem:code-cpx}, which states that for a contractible simplicial complex $\Sigma$, removing the relative interiors of some faces 
with contractible link, yields a space that is still contractible. 

We will use the following notation.
For a face $\tau$ of a simplicial complex $\Sigma$,
the {\em open star} of~$\tau$ is:
\[
{\rm St}_{\tau}(\Sigma) ~:= ~
\bigcup_{\tau \subseteq \sigma \in \Sigma} \interior ~\subseteq~ |\Sigma|~.
\]
(This union is in fact a disjoint union.)
If $v$ is a vertex of~$\Sigma$, we denote ${\rm St}_{\{v\}}(\Sigma)$ 
also by~${\rm St}_{v}(\Sigma)$. We denote the 
barycentric subdivision of $\Sigma$ by~$\Sigma'$, that is, the faces of $\Sigma'$ are flags 
$\tau_1 \subsetneq \tau_2 \subsetneq \dots \subsetneq \tau_k$ of 
nonempty
faces in~$\Sigma$. Equivalently, $\Sigma'$ is the order complex of the face poset of~$\Sigma$.
Finally, if $\tau$ is a face of $\Sigma$, we let $v(\tau)$ denote the corresponding vertex in
the barycentric subdivision~$\Sigma'$.

\begin{lem}
\label{lem:total-order}
	Let $\Sigma$ be a simplicial complex, and let $\tau_1, \tau_2, \dots, \tau_k$ be 
   nonempty
    faces of~$\Sigma$. Then the intersection 
	$\bigcap_i {\rm St}_{v(\tau_i)}(\Sigma')$ is nonempty if and only if the faces $\tau_i$ are totally ordered by inclusion.
\end{lem}

\begin{proof}
	If $\bigcap_i {\rm St}_{v(\tau_i)}(\Sigma') \ne \emptyset$, let $x \in \bigcap_i {\rm St}_{v(\tau_i)}(\Sigma')$. There is a unique
	face $\sigma' \in \Sigma'$ with $x \in {\rm int}(\sigma')$. Then ${\rm int}(\sigma') \subseteq {\rm St}_{v(\tau_i)}(\Sigma')$ for all $i \in \{1,2,\dots,k\}$.
	Thus, $v(\tau_1),v(\tau_2), \dots, v(\tau_k)$ are vertices of~$\sigma'$, and so form a face in $\Sigma'$. This means that the $\tau_i$'s,
    after reordering, form a flag in $\Sigma$, i.e., 
   are totally ordered by inclusion.
	
	Conversely, if, after reordering, $\tau_1 \subseteq \tau_2 \subseteq \dots \subseteq \tau_k$, then this flag     corresponds to a face $\sigma'$ of~$\Sigma'$
	with vertex set $v(\tau_1),v(\tau_2), \dots, v(\tau_k)$. Now ${\rm int}(\sigma') \subseteq {\rm St}_{v(\tau_i)}(\Sigma')$ for all $i \in \{1,2,\dots,k\}$, so
	${\bigcap_i {\rm St}_{v(\tau_i)}(\Sigma') \ne \emptyset}$.
\end{proof}

\begin{lem}
\label{lem:nerve-subdivision}
	Let $\Gamma$ be a collection of faces of a simplicial complex~$\Sigma$. Let $K = \bigcup_{\sigma \in \Gamma} \interior \subseteq |\Sigma|$. 
	Then for any $\tau \in \Gamma$, the intersection $K \cap {\rm St}_{v(\tau)}(\Sigma')$ is open in~$K$ and contractible. Moreover, the collection
    \[ \mathcal U ~=~ \left\{ K \cap {\rm St}_{v(\tau)}(\Sigma') \mid \tau \in \Gamma \right\} \]
is a good cover of $K$, and the nerve 
	$\mathcal N(\mathcal U)$ is homotopy equivalent to~$K$.
\end{lem}

\begin{proof}
	The open star ${\rm St}_{v(\tau)}(\Sigma')$ is open in~$|\Sigma|$ and thus $K \cap {\rm St}_{v(\tau)}(\Sigma')$ is open in~$K$.
	There is a deformation retraction of ${\rm St}_{v(\tau)}(\Sigma')$ to~$v(\tau)$ along straight lines. 
    Indeed, if $\sigma$ is a face of $\Sigma$ with
	$\tau \subseteq \sigma$, then for any $x \in \interior$ the segment from $v(\tau)$ to~$x$, excluding~$v(\tau)$, is entirely within~$\interior$.
	Thus the deformation retraction induces a deformation retraction of $K \cap {\rm St}_{v(\tau)}(\Sigma')$ to~$v(\tau)$.
    So, $K \cap {\rm St}_{v(\tau)}(\Sigma')$ is contractible.

To show that $\mathcal U$ is a cover of $K$, 
let $x \in K$.  Then $x \in {\rm int}(\sigma)$ for some $\sigma \in \Gamma$, and so $x \in {\rm int}(\tau')\subseteq {\rm St}_{v(\sigma)}(\Sigma')$ for some face $\tau'$ of $\Sigma'$ for which $v(\sigma)$ is a vertex.
Hence, $x$ is in the set $K \cap {\rm St}_{v(\sigma)}(\Sigma')$ from $\mathcal U$.

	If $\tau_1, \tau_2, \dots, \tau_k$ are faces in $\Gamma$ such that $\bigcap_i (K \cap {\rm St}_{v(\tau_i)}(\Sigma')) \ne \emptyset$, 
	then the faces $\tau_i$ are totally ordered by inclusion by Lemma~\ref{lem:total-order} and thus
	correspond to a face $\tau'$ of~$\Sigma'$. 
    It is straightforward to check that the intersection $\bigcap_i (K \cap {\rm St}_{v(\tau_i)}(\Sigma'))$ is equal to
	$K \cap {\rm St}_{\tau'}(\Sigma')$, which is contractible, since there is a straight-line deformation retraction to the barycenter of~$\tau'$,
	which does not leave~$K$. 
    Thus, $\mathcal U$ is a good cover, 
    and so the nerve theorem (Proposition~\ref{prop:nerve-thm})
implies that
	the nerve $\mathcal N(\mathcal U)$ is homotopy equivalent to~$K$. 
\end{proof}

\begin{prop} [Lemma~\ref{lem:code-cpx}, restated] \label{prop:lemma-restated}
Let $\Sigma$ be a contractible simplicial complex, 
and let $\sigma_1$, $\sigma_2$, \dots, $\sigma_k$ be faces of 
$\Sigma$ such that,
for  $i \in \{1,2,\dots,k\}$, the link ${\rm Lk}_{\sigma_i}(\Sigma)$ is contractible.
Then
$$|\Sigma| \setminus \left( 
{\rm int}(\sigma_1) \cup
{\rm int}(\sigma_2) \cup
\dots \cup
{\rm int}(\sigma_k)
\right)$$
is contractible.
\end{prop}
\begin{proof} 	Let $\Gamma$ be the collection of all faces of~$\Sigma$ with the exception of $\sigma_1, \sigma_2, \dots, \sigma_k$. 
	Let $K = \bigcup_{\sigma \in \Gamma} \interior = |\Sigma| \setminus \left({\rm int}(\sigma_1) \cup {\rm int}(\sigma_2) 
	\cup \dots \cup {\rm int}(\sigma_k)\right)$, and 
    let $ \mathcal U = \left\{ K \cap {\rm St}_{v(\tau)}(\Sigma') \mid \tau \in \Gamma \right\} $. 
    	    By Lemma~\ref{lem:nerve-subdivision}, the collection $\mathcal U$
	is a good cover of $K$, and 
    the nerve~$\mathcal N(\mathcal U)$ is homotopy equivalent to 
    $K$.
        
    Pick an inclusion-maximal face among 
	$\sigma_1, \sigma_2, \dots, \sigma_k$, say, $\sigma_1$ is not contained in any other~$\sigma_i$. Let $\widehat \Gamma 
	= \Gamma \cup \{\sigma_1\}$, and let $\widehat K = \bigcup_{\sigma \in \widehat\Gamma} \interior$. 
    Let $\widehat{\mathcal U} = \left\{ K \cap {\rm St}_{v(\tau)}(\Sigma') \mid \tau \in \widehat \Gamma \right\}$.
    		    Then, again by Lemma~\ref{lem:nerve-subdivision}, 
        $\widehat{\mathcal U}$
	is a good cover of $\widehat K$, and 
    the nerve $\mathcal N(\widehat{\mathcal U})$
	is homotopy equivalent to~$\widehat K$.
	
        To recap,
    $\mathcal N(\mathcal U) \simeq K$ and 
    $\mathcal N(\widehat{\mathcal U}) \simeq \widehat K$.
    Hence, the following claim would imply that $K \simeq \widehat K$:
   \\
   {\bf Claim:} $\mathcal N(\mathcal U) \simeq \mathcal N(\widehat{\mathcal U})$.\\
   Moreover, repeating this argument an additional $k-1$ times (applying it next to $\widehat K$, and so on) would imply that $K \simeq |\Sigma|$, and thus, as desired, $K$ 
	is contractible.  
    Hence, we need only prove the Claim.

	We first view the two nerves $\mathcal N(\mathcal U) $ and $ \mathcal N(\widehat{\mathcal U})$ as (abstract) simplicial complexes.    
Let ${\tau_1, \tau_2, \dots, \tau_\ell \in \Gamma}$. 
Then $\bigcap_i \left( K \cap {\rm St}_{v(\tau_i)}(\Sigma') \right) \ne \emptyset$
	if and only if $\bigcap_i \left( \widehat K \cap {\rm St}_{v(\tau_i)}(\Sigma') \right) \ne \emptyset$. 
    This is because the open vertex stars
	${\rm St}_{v(\tau_i)}(\Sigma')$,
    if they intersect, 
    cannot only intersect in a subset of ${\rm int}(\sigma_1)$, since the $\tau_i$'s are totally 
	ordered by inclusion by Lemma~\ref{lem:total-order},
    so the corresponding face $\tau'$ of $\Sigma'$ would have ${\rm int}(\tau')$ contained in the intersection
    $ \bigcap_i {\rm St}_{v(\tau_i)}(\Sigma')$, which would imply that one of the $\tau_i$'s is equal to $\sigma_1$, 
    and this would contradict the fact that the $\tau_i$'s are in $\Gamma$. 
    Thus the simplicial complexes 
	$\mathcal N(\mathcal U)$ and $\mathcal N(\widehat{\mathcal U})$ differ only in the vertex $w$ corresponding to the set 
	$\widehat K \cap {\rm St}_{v(\sigma_1)}(\Sigma')$ and 
    the faces incident to $w$.
    	
Let $\Delta_1$ denote the join ${\rm Lk}_w(\mathcal N(\widehat{\mathcal U})) * \{w \}$, and let $\Delta_2 = \mathcal N(\mathcal U) $.  
Then $\Delta_1$ is a cone over the vertex~$w$ and thus is contractible; also, 
$\Delta_1 \cap \Delta_2 = {\rm Lk}_w(\mathcal N(\widehat{\mathcal U}))$.  
So, to complete the proof it suffices to show that ${\rm Lk}_w(\mathcal N(\widehat{\mathcal U}))$ is contractible,  as~\cite[Lemma~10.3]{bjorner} will imply that
$\Delta_1 \cup \Delta_2 \simeq \Delta_2 $, that is, 
$\mathcal N(\widehat{\mathcal U})
\simeq 
\mathcal N(\mathcal U) 
$.

	Let $P$ be the poset of all faces in $\Gamma$ that are properly contained in $\sigma_1$, and denote its order
	complex by~$\Delta(P)$. We will show that ${\rm Lk}_w(\mathcal N(\widehat{\mathcal U}))$ is isomorphic to the join
	$\Delta(P) * ({\rm Lk}_{\sigma_1}(\Sigma))'$. A face of ${\rm Lk}_w(\mathcal N(\widehat{\mathcal U}))$ corresponds to
	faces $\tau_1, \tau_2, \dots, \tau_k \in \Gamma$ such that 
	$$\widehat K \cap {\rm St}_{v(\sigma_1)}(\Sigma') \cap \bigcap_i \widehat K \cap {\rm St}_{v(\tau_i)}(\Sigma') 
    ~\ne~ \emptyset~.$$
	Up to reordering the $\tau_i$'s, this implies, 
    by Lemma~\ref{lem:total-order}, that $\tau_1 \subseteq \dots \subseteq \tau_\ell \subseteq \sigma_1 \subseteq \tau_{\ell+1} \subseteq \dots \subseteq \tau_k$. The flag of faces $\tau_1 \subseteq \dots \subseteq \tau_\ell$ determines a face of~$\Delta(P)$, while
	$\tau_{\ell+1} \subseteq \dots \subseteq \tau_k$ determines a face of~$({\rm Lk}_{\sigma_1}(\Sigma))'$. Conversely, a face of 
	$\Delta(P) * ({\rm Lk}_{\sigma_1}(\Sigma))'$ corresponds to a flag of faces $\tau_1 \subseteq \dots \subseteq \tau_\ell \subseteq \sigma_1$
	with $\tau_1, \dots, \tau_\ell \in \Gamma$ and a flag $\sigma_1 \subseteq \tau_{\ell+1} \subseteq \dots \subseteq \tau_k$ with 
	$\tau_{\ell+1}, \dots, \tau_k \in \Sigma$. 
    We chose $\sigma_1$ as an inclusion-maximal face of $\Sigma$ not in~$\Gamma$, so $\tau_{\ell+1}, \dots, \tau_k \in \Gamma$. Thus we have inclusions 
	$\tau_1 \subseteq \dots \subseteq \tau_\ell \subseteq \sigma_1 \subseteq \tau_{\ell+1} \subseteq \dots \subseteq \tau_k$ with $\tau_1, \dots, \tau_k \in \Gamma$,
	which precisely corresponds to a face of~${\rm Lk}_w(\mathcal N(\widehat{\mathcal U}))$ by Lemma~\ref{lem:total-order}.
    Hence, ${\rm Lk}_w(\mathcal N(\widehat{\mathcal U}))$ is isomorphic to
	$\Delta(P) * ({\rm Lk}_{\sigma_1}(\Sigma))'$.
	
	By hypothesis, $({\rm Lk}_{\sigma_1}(\Sigma))'$ is contractible, and thus the join $\Delta(P) * ({\rm Lk}_{\sigma_1}(\Sigma))'$ is contractible.
    This join, we saw, is isomorphic to ${\rm Lk}_v(\mathcal N(\widehat{\mathcal U}))$, 
    so ${\rm Lk}_v(\mathcal N(\widehat{\mathcal U}))$
    is contractible.  Hence, as explained above, 
    \cite[Lemma~10.3]{bjorner} implies that $\mathcal N(\mathcal U) \simeq 
	\mathcal N(\widehat{\mathcal U})$, proving the Claim.
    \end{proof}

\section{Two definitions of collapsible} 
Here we prove that the definition of ``collapsible'' given earlier (Definition~\ref{def:d-collapse}(2)) is equivalent to Definition~\ref{def:collapsible-alternate} below, which is also used in the literature~\cite{bjorner,cohen}.  The proof was conveyed to us by Martin Tancer.

Recall that for a simplicial complex $\Delta$, an {\em elementary collapse} induced by a 
pair $(\sigma,\tau)$, where $\sigma$ is a face
contained in a unique facet $\tau$, is the simplicial complex 
$\Delta' = \Delta \setminus \{\nu \in \Delta \mid \sigma \subseteq \nu\}$.

\begin{defn} \label{def:collapsible-alternate}
A \textbf{collapse} of $\Delta$ to $\Delta'$ is a sequence of elementary collapses
\underline{induced by pairs $(\sigma,\tau)$ with}
\underline{$\dim \sigma = \dim \tau - 1$}
starting with $\Delta$ and ending with $\Delta'$. 
A simplicial complex is \textbf{collapsible} if it collapses to a point (via some sequence).  
\end{defn}

\begin{prop} \label{prop:2-collaps-defns-equivalent}
The definitions of ``collapsible'', in Definitions~\ref{def:d-collapse}(2) and~\ref{def:collapsible-alternate}, respectively, are equivalent.
\end{prop}

\begin{proof}  Let $\Delta$ be a simplicial complex.  
We need only show that every elementary collapse $\Delta \to \Delta'$ as in Definition~\ref{def:d-collapse}(2), i.e.,
induced by some $(\sigma,\tau)$ with $\sigma \subsetneq \tau$ and $\tau$ a facet,
can be obtained by a sequence of elementary collapses as in Definition~\ref{def:collapsible-alternate}.
We proceed by induction on $(\dim \tau - \dim \sigma)$.  The base case, when $\dim \tau - \dim \sigma=1$,
is already an elementary collapse under either definition.

For the inductive step, assume that $\dim \tau - \dim \sigma\geq 2$ and 
that every 
elementary collapse (of any simplicial complex) induced by pairs $(\sigma', \tau')$ with 
$(\dim \tau' - \dim \sigma') < (\dim \tau - \dim \sigma)$, can be obtained by a sequence of 
elementary collapses as in Definition~\ref{def:collapsible-alternate}.

Let $z \in \tau \setminus \sigma$.
Then $\sigma \cup \{z\} \subsetneq \tau$, and also
$\tau$ is the unique facet containing $\sigma \cup \{z\}$.  
Therefore, by the inductive hypothesis, the collapse $\Delta \to \widetilde \Delta$ induced by $(\sigma \cup \{z\},\tau)$
can be obtained by a sequence of elementary collapses as in Definition~\ref{def:collapsible-alternate}. Thus, we need only show that 
we can apply such elementary collapses
to 
$\widetilde \Delta$ to obtain  
$\Delta'$ (the outcome of the $(\sigma,\tau)$-collapse).

Note that $\sigma$ and $\tau \setminus \{z\}$ are both faces of 
$\widetilde \Delta$ (only faces of $\Delta$ containing $\sigma \cup \{z\}$ were removed).  
We claim that the unique facet of $\widetilde \Delta$ that contains $\sigma$ is 
$\tau \setminus \{z\}$.
Indeed, first, $\tau \setminus \{z\}$
contains $\sigma$ (by our choice of $\sigma$, $\tau$, and $z$), and, second, $\tau \setminus \{z\}$ is maximal in $\widetilde \Delta$ (otherwise $\tau$ would {\em not} have been maximal in $\Delta$).

Thus, by the inductive hypothesis,
the collapse $\widetilde \Delta \to \hat \Delta$
induced by $(\sigma,\tau\setminus \{z\})$
can be obtained by a sequence of elemetary collapses as in Definition~\ref{def:collapsible-alternate}.
So, to complete the proof, we need only show that $\hat \Delta = \Delta'$.  Indeed, 
$\widetilde \Delta$ was obtained by removing faces of $\Delta$ that contain $\sigma \cup \{z\}$, and then 
$\hat \Delta$ was obtained by removing faces that contain $\sigma$;
hence, 
$\hat \Delta = \Delta \setminus \{ \nu \in \Delta \mid \sigma \subseteq \nu \}
 =
  \Delta' $.
\end{proof}

\end{document}